\documentclass{article}

\usepackage{graphicx} 
\usepackage[utf8]{inputenc}
\usepackage[T1]{fontenc}
\usepackage[margin=1in]{geometry}
\usepackage{amsmath,amsfonts,amssymb,amsthm,mathrsfs}
\usepackage{bm}
\usepackage{microtype}

\usepackage{graphicx}
\usepackage{float}
\usepackage[table]{xcolor}
\usepackage[para,online,flushleft]{threeparttable}
\usepackage{booktabs}
\usepackage{enumitem}
\usepackage{algorithm}
\usepackage{algorithmic}
\usepackage{placeins}
\usepackage{bbding}
\usepackage{makecell}
\usepackage{hhline}
\usepackage[sort]{natbib}

\usepackage[
    colorlinks=true,
    linkcolor=blue,
    citecolor=blue,
    urlcolor=blue
]{hyperref}

\theoremstyle{plain}
\newtheorem{theorem}{Theorem}[section]
\newtheorem{lemma}[theorem]{Lemma}

\theoremstyle{definition}
\newtheorem{definition}[theorem]{Definition}

\usepackage{amsmath}\allowdisplaybreaks
\usepackage{amsfonts,bm}
\usepackage{amssymb}

\def\eqref#1{equation~(\ref{#1})}
\def\1{\bf{1}}

\newcommand{\Norm}[1]{\left\| #1 \right\|}

\def\inner#1#2{\left\langle #1, #2 \right\rangle}

\def\vzero{{\bf{0}}}

\def\vg{{\bf{g}}}
\def\vh{{\bf{h}}}

\def\vk{{\bf{k}}}

\def\vu{{\bf{u}}}
\def\vv{{\bf{v}}}

\def\vx{{\bf{x}}}
\def\vy{{\bf{y}}}

\def\fA{{\mathcal{A}}}
\def\fB{{\mathcal{B}}}

\def\fD{{\mathcal{D}}}

\def\fF{{\mathcal{F}}}

\def\fO{{\mathcal{O}}}
\def\fP{{\mathcal{P}}}

\def\fZ{{\mathcal{Z}}}

\def\BE{{\mathbb{E}}}
\def\BN{{\mathbb{N}}}

\def\BR{{\mathbb{R}}}

\def\mI {{\bf I}}

\def\mU {{\bf U}}

\usepackage{amsthm}
\theoremstyle{plain}
\newtheorem{remark}{Remark}

\makeatletter
\def\Ddots{\mathinner{\mkern1mu\raise\p@
\vbox{\kern7\p@\hbox{.}}\mkern2mu
\raise4\p@\hbox{.}\mkern2mu\raise7\p@\hbox{.}\mkern1mu}}
\makeatother

\makeatletter
\newcommand*{\rom}[1]{\expandafter\@slowromancap\romannumeral #1@}
\makeatother

\title{On the Complexity of Finding Stationary Points in Stochastic $(L_0, L_1)$-Smooth Nonconvex Optimization}

\author{
Yanpeng Liang \qquad\qquad\qquad Luo Luo
}

\date{}

\begin{document}

\maketitle

\begin{abstract}%
We study the stochastic first-order oracle (SFO) complexity of finding $\epsilon$-stationary points of $(L_0,L_1)$-smooth nonconvex functions. 
Under bounded-variance noise, existing analyses of normalized and clipped stochastic gradient descent (SGD) establish the optimal dependence on $\epsilon$ and~$L_0$, 
but the jointly optimal dependence on $L_0$, $L_1$, and $\epsilon$ remains unclear. 
We resolve this question for SFO algorithms with an unbiased gradient estimate whose noise has a  $p$th central moment bounded by $\sigma^p$ for some $\sigma\geq 0$, where $p\in(1,2]$.
Specifically, we establish an SFO lower bound of~$\Omega\big(\big(\sigma^{\frac{p}{p-1}}\epsilon^{-\frac{p}{p-1}}+1\big)
(L_0\Delta\epsilon^{-2}+L_1\Delta\epsilon^{-1})\big)$,
where~$\Delta>0$ is the initial optimality gap.
This bound identifies an unavoidable contribution proportional to  $L_1\epsilon^{-1}\Delta$. 
We also prove that mini-batch (smoothed) normalized SGD can attain a matching upper bound, thereby establishing the optimal SFO complexity.
Furthermore, we consider the hardness of finite-sum optimization, showing that the individual $(L_0,L_1)$-smoothness alone admits no finite, dimension-independent incremental first-order oracle complexity bound in the worst case.
\end{abstract}

\section{Introduction}

Empirical observations from language model training suggest that the local smoothness parameter of the objective increases with the gradient norm \citep{zhang2020gradient,zhang2020improved}.
This motivates us to study  nonconvex optimization problem 
\begin{equation}
    \min_{\vx\in\BR^d} f(\vx),
    \label{eq:problem}
\end{equation}
where $f:\BR^d\to\BR$ is $(L_0,L_1)$-smooth, i.e., there exist $L_0,L_1\geq 0$ such that 
\begin{equation}\label{eq:l0l1-intro}
    \Norm{\nabla f(\vx)-\nabla f(\vy)}
    \le
    \left(
        L_0+L_1\Norm{\nabla f(\vx)}
    \right)
    \Norm{\vx-\vy}
\end{equation}
for all $\vx,\vy\in\BR^d$ with $\Norm{\vx-\vy}\le 1/L_1$, where we define $1/L_1=+\infty$ if $L_1=0$. 
Unlike the standard smooth condition, the term $L_1\Norm{\nabla f(\vx)}$ on the right-hand side of inequality (\ref{eq:l0l1-intro}) indicates that the gradient of the objective function may not be Lipschitz continuous, which introduces additional challenges.
In recent years, numerous works have studied the theory and algorithms for optimization problems under the $(L_0, L_1)$-smooth condition \citep{chen2023generalized,vankov2024optimizing,li2024problem,li2023convergence,gaash2026convergence,gorbunov2025methods,koloskova2023revisiting,chezhegov2025convergence,tovmasyan2025revisiting,crawshaw2025complexity,lobanov2024linear,tyurin2024toward,reisizadeh2025variance,xie2024trust,zhang2020improved,pmlr-v306-yu26s,khirirat2024error,jiang2025decentralized,borodich2026nesterov,luo2026decentralized,hubler2024gradient,bo2026provably,bo2026linear}. 
However, existing results only establish the optimal dependence on $L_0$ but the optimality in terms of $L_1$, which is specific to this class of problems, remains unclear.

We focus on the complexity of finding an $\epsilon$-stationary point of problem (\ref{eq:problem}). 
We consider the complexity bound in the computation model that algorithms can access the information of the objective function through a stochastic first-order oracle (SFO) \citep{nemirovskij1983problem,arjevani2023lower,zhang2020adaptive}.
Specifically, for a given point $\vx\in\BR^d$, the SFO oracle can draw~$Z$ from distribution $\fD$ and return the gradient estimate $\vg(\vx;Z)$ satisfying 
\begin{equation}\label{eq:oracle-condition}
\begin{aligned}
    \BE_Z[\vg(\vx;Z)] = \nabla f(\vx)
    \qquad\text{and}\qquad
    \BE_Z\left[
    \Norm{\vg(\vx;Z)-\nabla f(\vx)}^p
    \right] \leq \sigma^p,
\end{aligned}
\end{equation}
where $\sigma\geq 0$ and $p\in(1,2]$.
Note that condition (\ref{eq:oracle-condition}) only requires the gradient estimate to satisfy
the $p$th bounded central moment ($p$-BCM), which is more general than the bounded variance assumption and allows heavy-tailed gradient noise that is widely observed in machine learning applications \citep{battash2024revisiting,ahn2023linear,garg2021proximal,cayci2023provably,zhu2024robust,gurbuzbalaban2025heavy,gorbunov2020stochastic,gorbunov2022clipped,sadiev2023high,gorbunov2024high,zhang2020adaptive,nguyen2023high,liu2023stochastic,hubler2024gradient,liu2025nonconvex,sun2024gradient,he2025complexity,liu2025stochastic,wu2026optimal,chen2026stability,yu2025decentralized,hubler2024parameter,wang2026near}.
We~also impose the standard assumption of a bounded initial optimality gap, i.e., for the algorithm with initial point~$\vx_0\in\BR^d$, there exists some $\Delta>0$ such that
\begin{align}\label{eq:optimal-gap}
    f(\vx_0) - \inf_{\vx\in\BR^d} f(\vx) \leq \Delta.
\end{align}
Existing studies on the lower complexity bound for $(L_0, L_1)$-smooth nonconvex optimization are restricted to specific algorithms \citep{zhang2020adaptive,hubler2024gradient,crawshaw2025complexity}.
For example, \citet{zhang2020gradient} showed that gradient descent (GD) methods with fixed stepsize cannot avoid the iteration complexity of $\Omega(L_1\Delta\epsilon^{-2})$;
\citet{hubler2024parameter} showed that normalized gradient descent with momentum (NGD-M) and exponential decaying stepsize cannot achieve
a parameter-agnostic convergence without an exponential dependence on $L_1$;
\citet{crawshaw2025complexity} studied lower bounds 
for AdaGrad \citep{duchi2011adaptive} and its variants \citep{li2019convergence,wang2023convergence}.
In concurrent works, \citet{rudomanov2026descending,rudomanov2026bdescending} establish the lower complexity bounds for convex optimization under the $(L_0,L_1)$-smoothness by considering the linear
span algorithms \cite{nemirovskij1983problem}. 
However, the lower complexity bound of general first-order methods for 
$(L_0,L_1)$-smooth nonconvex optimization remains unexplored, even in the deterministic setting.

In this paper, we provide a hard instance $(f, \vg)$ satisfying conditions (\ref{eq:l0l1-intro}), (\ref{eq:oracle-condition}), and (\ref{eq:optimal-gap}) to show that every randomized algorithm requires at least
\begin{align}\label{eq:main-lower-bound}
    \Omega\left(\left(\frac{L_0\Delta}{\epsilon^2}+\frac{L_1\Delta}{\epsilon}\right)\left(\frac{\sigma^{\frac{p}{p-1}}}{\epsilon^{\frac{p}{p-1}}}+1\right)\right)
\end{align}
SFO calls to find a point $\hat\vx\in\BR^d$ such that $\BE[\Norm{\nabla f(\hat\vx)}]\leq \epsilon$. 
Here, the term $L_1\Delta\epsilon^{-1}$ captures the additional difficulty of the $(L_0, L_1)$-smooth optimization compared with the standard smooth setting.
In the case of $L_1=0$, the result of (\ref{eq:main-lower-bound}) recovers the tight lower complexity bound 
for standard smooth optimization \citep{zhang2020adaptive,liu2025nonconvex}.
Our hard instance for the lower bound follows the structure of  a probabilistic zero-chain function \citep{arjevani2023lower} to guarantee that each iteration with an SFO call can only reveal at most one relevant coordinate with low probability.
Different from the standard smooth case, our construction introduces the components of piecewise-defined function and hyperbolic function to characterize the specific property of the~$(L_0, L_1)$-smoothness, i.e., the curvature decreases as the gradient norm decreases, which leads to the dependence on $L_1\Delta\epsilon^{-1}$ in the lower bound~(\ref{eq:main-lower-bound}).
We also provide sharp analysis to show that mini-batch (smoothed) normalized SGD can exactly attain a matching upper bound,  
verifying the tightness of our results.
Furthermore, we show that the finite-sum optimization with individual $(L_0, L_1)$-smoothness alone cannot achieve finite,
dimension-independent incremental first-order
oracle upper complexity bound in the worst case. 

%{\color{red}We compare our results with related work in Table xxx}.

\paragraph{Paper Organization}
In Section \ref{sec:Preliminaries}, we formalize notation and computation model for our lower bound construction. 
In Section \ref{sec:L0-L1-smooth}, we construct the hard instance of $(L_0, L_1)$-smooth function. 
In Section \ref{sec:lower-bound}, we provide the lower bound for randomized algorithms with SFO under $p$-BCM assumption. 
In Section \ref{sec:upper-bound}, we analyze mini-batch (smoothed) normalized stochastic gradient descent to achieve the tight upper complexity bound. 
In Section~\ref{sec:finite-sum}, we show the hardness of finite-sum optimization under individual $(L_0, L_1)$-smoothness. 
Finally, we conclude our work in Section \ref{sec:conclusion}.

\section{Preliminaries}\label{sec:Preliminaries}

In this section, we introduce notations and problem setting for the lower bound analysis.

\subsection{Notations}

For a given vector $\vx=[x_1,\dots,x_d]\in\BR^{d}$, we denote
\begin{align}
    \mathrm{supp}(\vx):=\{i: x_i \ne 0\}.
\end{align}
For a given $\alpha\in[0,1)$, we define the progress for vector $\vx\in\BR^d$ as 
\begin{align*}
\mathrm{prog}_{\alpha}(\vx):=\max\{i\in\BN : |x_i| > \alpha\},    
\end{align*}
where we let~$x_0 \equiv 1$ for convention.
For $d\geq T$, we let $\operatorname{St}(d,T):=\left\{\mU\in\BR^{d\times T}:\mU^\top\mU=\mI_T\right\}$.
For a twice continuously differentiable mapping
$\Phi:\BR^d\to\BR^d$, we denote its Fr\'echet derivative and second Fr\'echet derivative as $D\Phi(x)\in\mathcal L(\mathbb R^d,\mathbb R^d)$ and $D^2\Phi(x)\in\mathcal L^2(\mathbb R^d;\mathbb R^d)$, respectively, and the corresponding induced operator norms at $\vx\in\BR^d$ are defined as 
\begin{align*}
    \Norm{D\Phi(\vx)}
    :=\sup_{\Norm{\vh}=1}\Norm{D\Phi(\vx)[\vh]}  \qquad\text{and}\qquad
    \Norm{D^2\Phi(\vx)}
    :=\sup_{\Norm{\vh}=\Norm{\vk}=1}
    \Norm{D^2\Phi(\vx)[\vh,\vk]}.
\end{align*}

\subsection{Problem Setting}

We focus on finding  $\epsilon$-stationary points of $(L_0, L_1)$-smooth function by unbiased stochastic gradient evaluation. 
Our problem setting follows the setup of \citet[Section 2]{arjevani2023lower} but relaxes the bounded variance assumption to $p$-BCM condition, which is formalized as follows.

\paragraph{The function class}

We study the complexity bound for finding $\epsilon$-stationary points of functions in the function class
\begin{align*} 
&\mathcal{F}(L_0,L_1,\Delta) 
= \Big\{f:\BR^d\to\BR: f(\vzero)- \inf_{\vx\in\BR^d} f(\vx)\leq\Delta,  \\
&~~~~~~~~\Norm{\nabla f(\vx)-\nabla f(\vy)}  \le  \left( L_0+L_1\Norm{\nabla f(\vx)} \right) \Norm{\vx-\vy}
~~\text{for all}~~\vx,\vy\in\BR^d \text{ with }\Norm{\vx-\vy}\le\frac{1}{L_1} \Big\}.
\end{align*}

\paragraph{The oracle class}
We consider the stochastic first-order oracle (SFO) from the oracle class $\fO(\sigma,p)$, where each oracle ${\mathsf O}_f\in\fO(\sigma,p)$ consists of a distribution $\fD$ on a measurable space~$\fZ$ and a measurable mapping
\mbox{${\mathsf O}_f(\vx;Z)=\vg(\vx;Z)\in\BR^d$}
such that for all $\vx\in\BR^d$, if $Z\sim\fD$, then it holds $\BE_Z[\vg(\vx;Z)]=\nabla f(\vx)$ 
and
$\BE_Z\left[\Norm{\vg(\vx;Z)-\nabla f(\vx)}^p\right] \leq \sigma^p$ with $p\in(1,2]$.

\paragraph{The algorithm class}
We focus on the randomized algorithm class $\fA_{\mathrm {rand}}$, where each~$\mathsf A\in\fA_{\mathrm {rand}}$ consists of a distribution $\fP_r$ on a measurable space $\mathcal{R}$ of measurable mappings
$\{\mathsf{A}^{(t)}\}_{t\in\BN}$ such that each~$\mathsf{A}^{(t)}$ takes in the first $t-1$ responses and the random seed $r\in\BR$ to produce the $t$th query.
We let~$\big\{\vx_{\mathsf{A}[{\mathsf O}_f]}^{(t)}\big\}_{t\in\BN}$
denote the random sequence of queries resulting from the algorithm
$\mathsf{A}$ with oracle~${\mathsf O}_f$, defined recursively as
\begin{equation*}
\begin{split}    
\vx_{\mathsf{A}[{\mathsf O}_f]}^{(t)}
    =
    \mathsf{A}^{(t)}
    \Big(
        r,\,
        {\mathsf O}_f
        \big(
            \vx_{\mathsf{A}[{\mathsf O}_f]}^{(1)};
            Z^{(1)}
        \big),
        \ldots,
        {\mathsf O}_f
        \big(
            \vx_{\mathsf{A}[{\mathsf O}_f]}^{(t-1)};
            Z^{(t-1)}
        \big)
    \Big).
\end{split}
\end{equation*}
Without loss of generality, we suppose the initial point of the algorithm is $\vx_{\mathsf{A}[{\mathsf O}_f]}^{(0)}=\vzero$.

\paragraph{Complexity measures}
Let $\mathcal{P}[\mathcal{F}(L_0,L_1,\Delta)]$ be the set of all distributions over the function class~$\mathcal{F}(L_0,L_1,\Delta)$.
We consider the distributional complexity as follows
\begin{equation}
\begin{aligned}
   \!\!\!\!m_{\epsilon}^{\mathrm{rand}}
    \left(L_0,L_1,\Delta,\sigma,p
    \right) 
  :=\sup_{{\mathsf O}\in\mathcal{O}(\sigma, p)}\;
    \sup_{P_f\in\mathcal{P}[\mathcal{F}(L_0,L_1,\Delta)]}
  \inf_{\mathsf{A}\in\mathcal{A}_{\mathrm{rand}}}\;
    \inf\left\{ T\in\BN:
             \BE\left[\Big\|\nabla f\big(\vx_{\mathsf{A}[\mathsf{O}_f]}^{(T)}\big)\Big\|\right]
           \le \epsilon
        \right\}.
\end{aligned}
\label{eq:distributional-complexity}
\end{equation}
Note that if we can prove that 
$m_{\epsilon}^{\mathrm{rand}}\left(L_0,L_1,\Delta,\sigma,p\right)>T$, 
then there exists oracle ${\mathsf O}\in\fO(\sigma,p)$ such that for every ${\mathsf A}\in\fA_{\mathrm rand}$ there exists a function $f\in \fF(L_0, L_1, \Delta)$ such that $\BE\big[\big\|\nabla f\big(\vx_{\mathsf{A}[\mathsf{O}_f]}^{(T)}\big)\big\|\big]>\epsilon$.

\section{Construction for $(L_0,L_1)$-Smooth Function}\label{sec:L0-L1-smooth}

We first introduce the piecewise polynomial function $\psi$ and the hyperbolic function $\phi$ as follows
\begin{equation}
    \psi(t):=S(2t-1)   \quad \text{and} \quad   \phi(t):=1+\tanh(t),
    \label{eq:gate-functions}
\end{equation}
respectively, where $S:\BR\to\BR$ is defined as
\begin{equation}
    S(u)= \begin{cases}
        0, & u\le 0,\\
        10u^3-15u^4+6u^5, & u\in(0,1),\\
        1, & u\ge 1.
    \end{cases}
    \label{eq:gate-S}
\end{equation}

We can verify  functions $\psi$ and $\phi$ have the following properties.

\begin{lemma}
\label{lem:gate-properties}
The functions $\psi$ and $\phi$ defined in  
\eqref{eq:gate-functions} satisfy the following
properties. 
\begin{enumerate}[label=(\alph*),leftmargin=0.65cm,topsep=-0.05cm,parsep=-0.05cm]
    \item The function $\psi:\BR\to\BR$ satisfies $\psi(t)=0$ for all $t \le 1/2$, and 
    $\psi(t)=1$ for all~$t\ge1$.

    \item The function $\psi:\BR\to\BR$ satisfies
    $0\le\psi(t)\le1$, $0\le\psi'(t)\le4$, and $|\psi''(t)|\le24$ for all $t\in\BR$.

    \item The function $\phi:\BR\to\BR$ satisfies
    $0<\phi(t)<2$, $\phi'(t)=\operatorname{sech}^2(t)$, and   $|\phi''(t)|\le2\phi'(t)$ for all $t\in\BR$.

    \item The derivative $\phi':\BR\to\BR$ is even and 
    $\phi'(t)>2/5$ for all $t\in[-1, 1]$.
    % \[
    % \phi'(t)>\frac25
    % \qquad \text{for } |t|\le1.
    % \]
\end{enumerate}
\end{lemma}

\begin{remark}
    We illustrate the derivatives of $\psi$ and $\phi$ in Figure \ref{figure:component}. Note that the derivative $\psi'(t)$ is non-zero only for $t\in(1/2, 1)$, which is useful to construct the zero-chain. 
    Additionally, the magnitude of $\phi''(t)$  can always be bounded by $2\phi'(t)$, corresponding to the $(L_0, L_1)$-smoothness.
\end{remark}

\begin{figure*}[t]
\centering

\begin{minipage}[t]{0.32\textwidth}
    \centering
    \includegraphics[width=\linewidth]{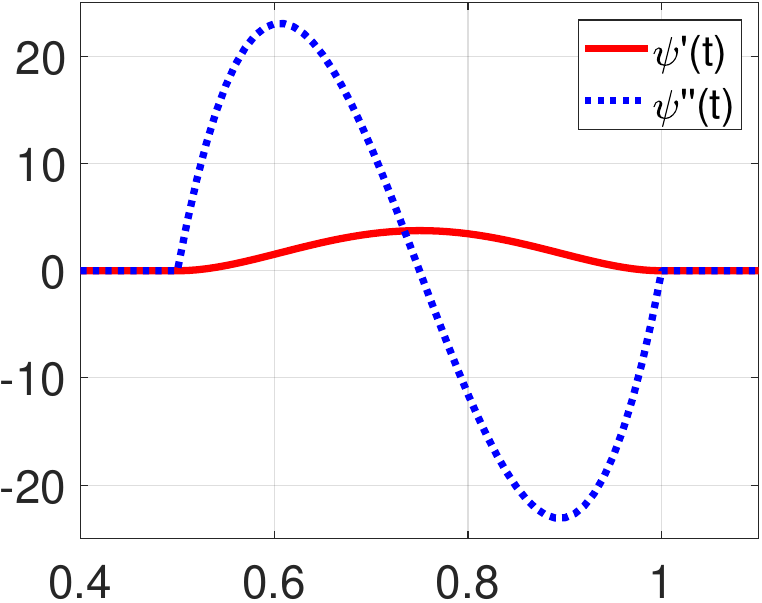} \\[0.15cm]
    
    \small (a) $\psi'(t)$ and $\psi''(t)$
\end{minipage}
\hfill
\begin{minipage}[t]{0.32\textwidth}
    \centering
    \includegraphics[width=\linewidth]{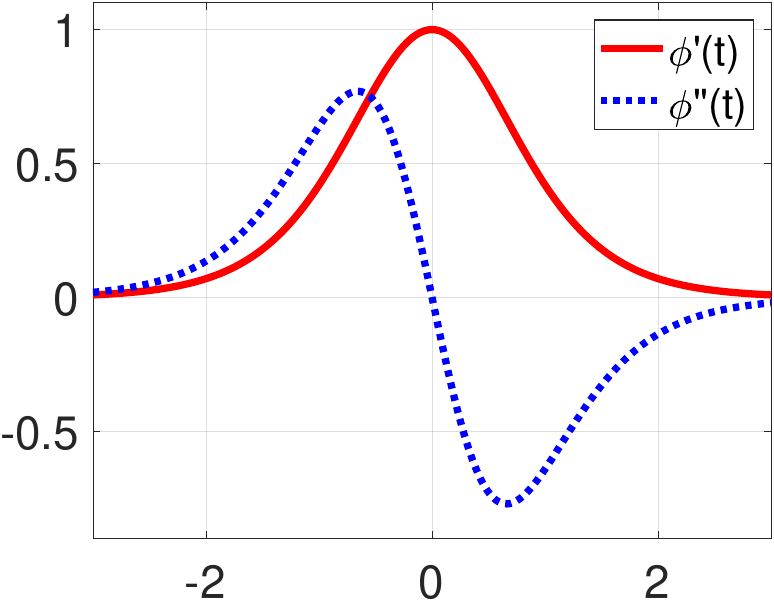} \\[0.15cm]
    
    \small (b) $\phi'(t)$ and $\phi''(t)$
\end{minipage}
\hfill
\begin{minipage}[t]{0.32\textwidth}
    \centering
    \includegraphics[width=\linewidth]{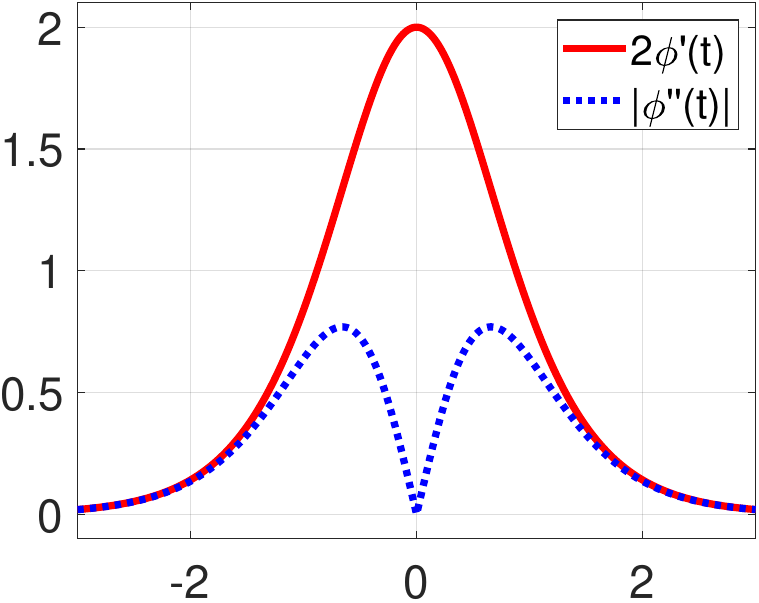} \\[0.15cm]
    
    \small (c) $2\phi'(t)$ and $|\phi''(t)|$
\end{minipage}

\caption{
First- and second-order derivatives of the component functions
$\psi$ and $\phi$.
}
\label{figure:component}
\end{figure*}

We then construct the function $\bar f_T:\BR^T\to\BR$ as follows
\begin{equation}
\begin{aligned}
&\bar f_T(\vy) := \sum_{i=1}^{T}A_i(\vy), 
~~~~\text{where}~~
A_i(\vy):=
\begin{cases}
-\phi(y_1), & \text{if } i=1, \\[0.1cm]
\psi(-y_{i-1})\phi(-y_i)
-\psi(y_{i-1})\phi(y_i), & \text{if } i\geq2.
\end{cases}
\end{aligned}
\label{eq:hard-zero-chain}
\end{equation}
Lemma \ref{lem:gate-properties} indicates the following properties of  $\bar f_T$.

\begin{lemma}%[Properties of the hard zero-chain]
\label{lem:hard-chain-properties}
The function $\bar f_T:\BR^T\to\BR$ defined in \eqref{eq:hard-zero-chain} satisfies the following properties.
\begin{enumerate}[label=(\alph*),leftmargin=0.65cm,topsep=-0.05cm,parsep=-0.05cm]
    \item The optimality gap holds 
    $\bar f_T(\vzero) - \inf_{\vy\in\BR^T}\bar f_T(\vy) =2T-1$.

    \item For all $\vy\in\BR^T$, it holds 
    $\Norm{\nabla\bar f_T(\vy)} \le9\sqrt T$
    and $\Norm{\nabla^2\bar f_T(\vy)} \le 58$.

    \item For a given $\vy=[y_1,\dots,y_T]\in\BR^T$, if there exists \mbox{$i\in[T]$} such that $|y_i|\le1$, then it holds $\Norm{\nabla\bar f_T(\vy)} \ge\frac25$.
    
    \item  For all $\vy\in\BR^T$, it holds 
    $\Norm{\nabla^2 \bar f_T(\vy)} \le 160\Norm{\nabla\bar f_T(\vy)}$.
    %$\Norm{\nabla^2 \bar f_T(\vy)} \le 160\Norm{\nabla\bar f_T(\vy)}$.

    \item  For all $\vy\in\BR^T$, it holds 
    $\operatorname{prog}_0(\nabla\bar f_T(\vy))\le\operatorname{prog}_{1/4}(\vy)+1$.
    %$\operatorname{prog}_0(\nabla\bar f_T(\vy))\le\operatorname{prog}_{1/4}(\vy)+1$.
\end{enumerate}
\end{lemma}

\begin{remark}
Lemma \ref{lem:hard-chain-properties}(d) implies that  function $\bar f_T$ is $(0, 320)$-smooth, which is essentially different from the analysis for standard smooth case \citep{carmon2020lower}.
In fact, the ratio of the gradient norm to the Hessian norm in the classical hard instance for standard smooth optimization is unbounded, so that it cannot characterize the specific hardness of $(L_0, L_1)$-smoothness.
Please refer to discussion in Appendix~\ref{appendix:disscusion} for details.
\end{remark}

We now follow \citet{carmon2020lower} to introduce the soft projection to a Euclidean ball centered at the origin with radius $R>0$, i.e.,
\begin{equation}
    \rho_R(\vx)
    :=
    \frac{\vx}{\sqrt{1+\Norm{\vx}^2/R^2}}.
    \label{eq:radial-compression}
\end{equation}
We focus on the case of $d\ge T$ and 
construct the function
\begin{equation}\label{eq:global-hard-function}
H_{\mU,\mu}(\vx)
:= \bar f_T\big(\mU^\top\rho_R(\vx)\big)
+\frac{\mu}{2}\Norm{\vx}^2, 
\end{equation}
where
$\mU\!\in\operatorname{St}(d,T)$,
$R=1000\sqrt{T}/\mu$, and
$\mu\in\left(0,1/100\right]$.

The following lemma shows the key properties for our hard instance.

\begin{lemma}
\label{lem:global-hard-properties}
The function $H_{\mU,\mu}:\BR^d\to\BR$ defined in
\eqref{eq:global-hard-function} satisfies
\begin{equation}
    H_{\mU,\mu}(\vzero)
    -\inf_{\vx\in\BR^d}H_{\mU,\mu}(\vx)
    \le2T.
    \label{eq:global-hard-gap}
\end{equation}
Moreover, if there exists some $i\in[T]$ such that
$\left| [\mU^\top\rho_R(\vx)]_i\right|\le1$,
then it holds
\begin{equation}
    \Norm{\nabla H_{\mU,\mu}(\vx)}
    \ge\frac14.
    \label{eq:global-hard-gradient}
\end{equation}
Furthermore,for all $\vx\in\BR^d$, it holds
\begin{equation}
    \Norm{\nabla^2H_{\mU,\mu}(\vx)}
    \le
    240\left(
    \Norm{\nabla H_{\mU,\mu}(\vx)}
    +\mu\log\frac1\mu
    \right).
    \label{eq:global-hard-curvature}
\end{equation}
\end{lemma}

We scale the function $H_{\mU,\mu}$ to achieve the hard instance $f_\mU:\BR^d\to\BR$ as 
\begin{equation}\label{eq:scaled-hard-function}
    f_{\mU}(\vx):=aH_{\mU,\mu}\left(\frac{\vx}{b}\right), 
\end{equation}
where $a,b>0$.
The following lemma shows we can achieve the desired $(L_0, L_1)$-smooth function for all $L_0,L_1>0$ by appropriate parameters setting.

\begin{lemma}\label{lem:hard-scaling}
For all $L_0>0,L_1\geq 0,\epsilon>0$, we use the definitions in equations (\ref{eq:gate-functions})--(\ref{eq:global-hard-function}) by taking
\begin{equation}\label{eq:hard-scaling-parameters}
\begin{split}
\mu:=\frac{\delta_0}{4\log({\mathrm e}/\delta_0)},\qquad a:=\frac{30208\epsilon^2}{L_0+L_1\epsilon},
\qquad\textit{and}\qquad
b:=\frac{1888\epsilon}{L_0+L_1\epsilon},
\end{split}
\end{equation}
where $\delta_0:=\min\left\{10^{-2},
{L_0}/{(16(L_0+L_1\epsilon))}\right\}$.
Then the scaled function $f_{\mU}(\vx):=aH_{\mU,\mu}(\vx/b)$ has the following properties.
\begin{enumerate}[label=(\alph*),leftmargin=0.65cm,topsep=-0.05cm,parsep=-0.05cm]
    \item The function $f_{\mU}:\BR^d\to\BR$ is $(L_0,L_1)$-smooth. 
    \item The optimality gap holds
    $f_{\mU}(\vzero)-\inf_{\vx\in\BR^d}f_{\mU}(\vx)\leq2aT$.

    \item For any $\vx\in\BR^d$, if there exists some $i\in[T]$ such that $\left|[\mU^\top\rho_R(\vx/b)]_i\right|\leq1$, then it holds $\Norm{\nabla f_{\mU}(\vx)}\geq4\epsilon$.
\end{enumerate}
\end{lemma}

\begin{remark}
Lemma~\ref{lem:hard-scaling} includes the case of $L_1=0$, 
so that it also works for the standard smooth setting.

\end{remark}

\section{The Lower Complexity Bound}\label{sec:lower-bound}

We establish a lower bound for randomized algorithms by constructing the following stochastic gradient evaluation which is probabilistic zero-chains \citep{arjevani2023lower}.

\begin{definition}[{\citet[Definition 2]{arjevani2023lower}}]
A stochastic gradient $\vg:\BR^T\times\fZ\to\BR^T$ is a probability-$\theta$ zero-chain with $\theta\in[0,1]$ if it holds
\begin{equation}
\begin{aligned}
\mathbb P_Z
\Big(
\forall\,\vy\in\BR^T:
\operatorname{prog}_0\big(\vg(\vy;Z)\big)
\le \operatorname{prog}_{1/4}(\vy)
\vg(\vy;Z)
=
\vg\big(
\Pi_{\operatorname{prog}_{1/4}(\vy)}(\vy); Z
\big)
\Big)
\ge 1-\theta
\end{aligned}
\label{eq:prob-zero-chain-no-progress}
\end{equation}
and
\begin{equation}
\begin{aligned}
\!\!\mathbb P_Z
\Big(
\forall\,\vy\in\BR^T:
\operatorname{prog}_0\big(\vg(\vy;Z)\big)
\le 1+\operatorname{prog}_{1/4}(\vy),
\vg(\vy;Z)
=
\vg\big(
\Pi_{1+\operatorname{prog}_{1/4}(\vy)}(\vy); Z
\big)
\Big)
=1.
\end{aligned}
\label{eq:prob-zero-chain-one-progress}
\end{equation}
Here, we denote 
\begin{align*}
\Pi_s(\vy)=[y_1,\dots,y_s,0,\dots,0]^\top     
\end{align*}
for a given $\vy=[y_1,\dots,y_T]^\top\in\BR^T$. In addition, we let $\Pi_s(\vy)=\vy$ when $s>T$ for convention.
\end{definition}

We introduce a stochastic evaluation for functions $\bar f_T$ as follows
\begin{equation}
    \bar F_T(\vy;Z)
    :=
    \bar f_T(\vy)
    +
    \left(\frac{Z}{q}-1\right)Q_T(\vy),
    \label{eq:stochastic-chain}
\end{equation}
where $Z\sim\mathrm{Bernoulli}(q)$ with $q\in(0,1]$ and the function $Q_T:\BR^T\to\BR$ is defined as
\begin{equation}
    Q_T(\vy):=\sum_{i=1}^{T}A_i(\vy)\left(\,\prod_{k=i}^{T}(1-S(4|y_{k}|-1))\right).
    \label{eq:masking-Q}
\end{equation}

Here, the functions $S$ and $A_i$ follow the definitions in
equations (\ref{eq:gate-S}) and (\ref{eq:hard-zero-chain}), respectively.
% Note that $\chi(t)=1$ for $|t|\le1/4$ and $\chi(t)=0$ for $|t|\ge1/2$.
The above constructions have the following properties.

\begin{lemma}\label{lem:information-hiding}
The functions $\bar F_T$ and $Q_T$ defined in equations (\ref{eq:stochastic-chain}) and (\ref{eq:masking-Q}) satisfy
\begin{enumerate}[label=(\alph*),leftmargin=0.65cm,topsep=-0.05cm,parsep=-0.05cm]
\item For all $\vy\in\BR^T$, it holds 
$\Norm{\nabla Q_T(\vy)}\le249$.
\item The stochastic gradient evaluation 
$\vg_T(\vy;Z):=\nabla\bar F_T(\vy;Z)$
is a probability-$q$ zero-chain and holds 
$\BE_Z[\vg_T(\vy;Z)]=\nabla\bar f_T(\vy)$.
\end{enumerate}
\end{lemma}

We define the scale stochastic first-order oracle for function $f_\mU$ as follows
\begin{equation}
\mathsf O_{f_{\mU}}(\vx;Z)
    :=\nabla F_{\mU}(\vx;Z),
    \label{eq:hard-stochastic-oracle}
\end{equation}
where
\begin{equation}
\!\!\!\!F_{\mU}(\vx;Z)
    := a\!\left(\bar F_T\left(\mU^\top\rho_R\left(\frac{\vx}{b}\right);Z\right) 
    + \frac{\mu}{2}\Norm{\frac{\vx}{b}}^2
    \right)\!.
\end{equation}
The parameters $a$, $b$, and $\mu$ follow the settings in Lemma~\ref{lem:hard-scaling}. The following lemma shows that the oracle $\mathsf O_{f_{\mU}}(\vx;Z)$ is unbiased and has a $p$th bounded central moment.

\begin{lemma}\label{lem:oracle-validity}
For a given $p\in(1,2]$ and $\sigma\ge0$, we let
\begin{equation}\label{eq:revelation-probability}
    q:=\begin{cases}
    1, & \sigma = 0, \\[0.25cm]
    \displaystyle{\min\left\{ 1,\, \left(\frac{2\cdot(3984)^p\epsilon^p}{\sigma^p}\right)^{\frac{1}{p-1}}
    \right\}}, & \sigma > 0.
    \end{cases}
\end{equation}
Then for all $\vx\in\BR^d$, it holds
\begin{equation}
\begin{aligned}
\BE_Z[\mathsf O_{f_{\mU}}(\vx;Z)]=\nabla f_{\mU}(\vx)
\qquad\text{and}\qquad
\BE_Z\left[\Norm{\mathsf O_{f_{\mU}}(\vx;Z)-\nabla f_{\mU}(\vx)}^p\right]\le \sigma^p.
\end{aligned}
\label{eq:oracle-p-moment}
\end{equation}
\end{lemma}

We now show that the above oracle yields the high-probability lower bound on gradient norm at points generated by randomized algorithms.

\begin{lemma}\label{lem:randomized-hardness}
%Let $\vg_T(\vy;Z)=\nabla\bar F_T(\vy;Z)$ be the probability-$q$ zero-chain established in Lemma~\ref{lem:information-hiding}.
We let $\delta\in(0,1)$ and $\mU\sim\operatorname{Unif}(\operatorname{St}(d,T))$, where
\begin{equation}\label{eq:rotation-dimension}
d\ge \left\lceil
T+32R^2\log\left(\frac{2T^2}{q\delta}\right)
\right\rceil.
\end{equation}
Then for any randomized algorithm $\mathsf A\in\mathcal A_{\mathrm{rand}}$
associated with the oracle $\mathsf O_{f_{\mU}}(\vx;Z)$ defined in \eqref{eq:hard-stochastic-oracle}, with probability at least $1-\delta$, it holds
\begin{equation}
    \Norm{\nabla f_{\mU}(\vx^{(t)})}
    \ge4\epsilon
\end{equation}
for all  $t\le(T-\log(2/\delta))/(2q)$,
where $\vx^{(t)}$ is the $t$th query generated by randomized algorithm $\mathsf A$. 
\end{lemma}

Based on Lemma \ref{lem:randomized-hardness}, we establish the  $L_1$-dependent lower bound as follows.

\begin{theorem}
\label{thm:L1-lower-bound}
For any given $p\in(1,2]$, there exists a numerical constant $\hat c>0$ such that for all $L_0>0$, $L_1\geq 0$, $\Delta\geq 0$, $\epsilon>0$, and $\sigma\geq0$, with $\epsilon\leq\hat c(\sqrt{L_0\Delta}+L_1\Delta)$, it holds
\begin{align}
\begin{split}
    m_{\epsilon}^{\mathrm{rand}}(L_0,L_1,\Delta,\sigma,p)
    \geq\Omega\left(
    \left(\frac{L_0\Delta}{\epsilon^2}+\frac{L_1\Delta}{\epsilon}\right)\left(\frac{\sigma^{\frac{p}{p-1}}}{\epsilon^{\frac{p}{p-1}}}+1\right)\right).
\end{split}
\label{eq:L1-lower-bound}
\end{align}
\end{theorem}

The term $L_1\Delta\epsilon^{-1}$ in \eqref{eq:L1-lower-bound} distinguishes the lower bound of $(L_0, L_1)$-smoothness from that of the standard smoothness.
This indicates that finding $\epsilon$-stationary points of $(L_0, L_1)$-smooth functions is indeed more difficult than in the standard smooth setting when $L_1 \geq \Omega(L_0\epsilon^{-1})$.

The SFO lower bound in Theorem \ref{thm:L1-lower-bound} connects several existing results as follows.
\begin{itemize}[leftmargin=0.65cm,topsep=-0.25cm,parsep=0.25cm]
    \item  In the case of $L_1=0$, the lower bound (\ref{eq:L1-lower-bound}) reduces to
    \begin{align*}
        \Omega\left(\frac{L_0\Delta}{\epsilon^2}\left(\frac{\sigma^{\frac{p}{p-1}}}{\epsilon^{\frac{p}{p-1}}}+1\right)\right),
    \end{align*}
    matching the form of optimal SFO complexity for standard $L_0$-smooth nonconvex optimization under the $p$-BCM condition. Moreover, unlike the lower bound by \citet[Theorem 3.3]{liu2025nonconvex}, which is restricted to zero-respecting algorithms,
    our result applies to arbitrary randomized first-order algorithms.
    
    \item In the case of $\sigma=0$, our lower bound  (\ref{eq:L1-lower-bound}) reduces to
    \begin{align*}        
    \Omega\bigg(\frac{L_0\Delta}{\epsilon^2}+\frac{L_1\Delta}{\epsilon}\bigg),
    \end{align*}
    matching best-known upper bounds for the exact first-order oracle complexity \citep{gorbunov2025methods,koloskova2023revisiting,pmlr-v267-tyurin25a,vankov2024optimizing}.
    In other words, this indicates the optimality of existing deterministic first-order methods.
\end{itemize}

\begin{remark}
Theorem~\ref{thm:L1-lower-bound} focuses on the lower bound in the case of 
$\epsilon\leq \fO(\sqrt{L_0\Delta}+ L_1\Delta)$.
This is because in the case of $\epsilon\geq\Omega(\sqrt{L_0\Delta}+ L_1\Delta)$, the initial point $\vx_0$ is already an~$\fO(\epsilon)$-stationary point of $f(\cdot)$. 
Please refer to Appendix \ref{appendix:accuracy} for detailed discussion.
\end{remark}

\section{Sharp Analysis for SFO Methods}\label{sec:upper-bound}

In this section, we refine the analysis of the mini-batch normalized stochastic gradient descent and its smoothed variants to establish the tight SFO upper complexity bound.

\subsection{Normalized Stochastic Gradient Descent}

Normalized stochastic gradient descent (NSGD) and its variants are widely used to solve nonconvex $(L_0, L_1)$-smooth optimization problems, as well as to address the heavy-tailed gradient noise \citep{luo2026decentralized,khirirat2024error,hubler2024gradient,liu2025nonconvex,chen2026stability,wu2026optimal,gorbunov2025methods}.
Somewhat surprisingly, no existing analysis has achieved the jointly optimal dependency on $L_0$, $L_1$, $\sigma$, and $\epsilon$, even for the bounded variance setting of~$p=2$.

We present (smoothed) mini-batch normalized stochastic gradient descent (MB-NSGD) in Algorithm \ref{alg:sn-mb-sgd}, which extends smoothed normalized gradient descent \citep{gorbunov2025methods} to the stochastic setting.
Note that by taking $\gamma=0$, Algorithm \ref{alg:sn-mb-sgd} reduces to  mini-batch normalized stochastic gradient descent \citep{hubler2024gradient}.
Here, we let $\vg/\Norm{\vg}=\vzero$ when $\vg=\vzero$ for convention.
In practice, we typically set a small $\gamma>0$ to improve the stability by avoiding the amplification of small perturbations when the norm of the (stochastic) gradient is close to zero.

\subsection{The Optimal Upper Complexity Bound}

The following theorem shows that Algorithm \ref{alg:sn-mb-sgd} with either $\gamma=0$ or $\gamma>0$ can find an $\epsilon$-stationary point of $(L_0,L_1)$-smooth function.

\begin{theorem}\label{thm:nsgd}
We run Algorithm \ref{alg:sn-mb-sgd} to solve problem (\ref{eq:problem}) with  $f\in\fF(L_0, L_1, \Delta)$ by using the stochastic first-order oracle $\vg:\BR^d\times\fZ\to\BR^d$ such that     
$\BE_Z[\vg(\vx;Z)]=\nabla f(\vx)$ 
and~$\BE_Z\left[\Norm{\vg(\vx;Z)-\nabla f(\vx)}^p
\right]\leq \sigma^p$
for some $p\in(1,2]$ and $\sigma\geq 0$.
By taking parameters
\begin{equation}\label{eq:nsgd-parameter}
\begin{split}
&\eta=\frac{\epsilon}{4(L_0+L_1\epsilon)},\qquad
\gamma\in\left[0,\frac{\epsilon}{4}\right],\qquad
b=\bigg\lceil
\bigg(\frac{32\sqrt2\,\sigma}{\epsilon}\bigg)^{\frac{p}{p-1}}+1
\bigg\rceil,\qquad\text{and}\qquad
T=\left\lceil\frac{4\Delta}{\eta\epsilon}\right\rceil,
\end{split}
\end{equation}
the output satisfies $\BE[\Norm{\nabla f(\hat\vx)}]\leq\epsilon$.
\end{theorem}

For a given $p\in(1,2]$, Theorem \ref{thm:nsgd} shows that both mini-batch normalized stochastic gradient descent (Algorithm~\ref{alg:sn-mb-sgd} with $\gamma=0$) and its smoothed variant (Algorithm \ref{alg:sn-mb-sgd} with $\gamma>0$) can attain the upper complexity bound of 
\begin{align*}
    Tb = \fO\left(\left(\frac{L_0\Delta}{\epsilon^2}+\frac{L_1\Delta}{\epsilon}\right)\left(\frac{\sigma^{\frac{p}{p-1}}}{\epsilon^{\frac{p}{p-1}}}+1\right)\right),
\end{align*}
matching the lower bound provided by Theorem~\ref{thm:L1-lower-bound}.

\begin{algorithm}[t]
  \caption{(Smoothed) MB-NSGD}
  \label{alg:sn-mb-sgd}
  \begin{algorithmic}
    \STATE {\bfseries Input:} $\vx_0\in\BR^d$, $\eta>0$, $\gamma\geq 0$, $b,T\in\BN$ 
    \\[0.1cm]
    \FOR{$t=0$ {\bfseries to} $T-1$}
      \STATE Sample $Z_{t,1},\dots,Z_{t,b}\overset{\mathrm{i.i.d.}}{\sim}\fZ$ \\[0.1cm]
      \STATE $\displaystyle \vg_t=\frac{1}{b}\sum_{i=1}^b \vg(\vx_t;Z_{t,i})$ \\[0.1cm]
      \STATE $\displaystyle \vx_{t+1}=\vx_t-\frac{\eta\vg_t}{\|\vg_t\|+\gamma}$ \\[0.1cm]
    \ENDFOR \\[0.1cm]
    \STATE {\bfseries Output:} $\hat\vx\sim\mathrm{Unif}(\{\vx_0,\dots,\vx_{T-1}\})$
  \end{algorithmic}
\end{algorithm}

\subsection{Related Work}

We summarize main results of this paper and related work in Table \ref{table:compare}.  
To the best of our knowledge, no prior SFO upper bound for stochastic nonconvex $(L_0,L_1)$-smooth optimization could match lower bound (\ref{eq:L1-lower-bound}) for general $L_1>0$ and $\sigma>0$.
In a recent work, \citet{liu2025nonconvex} studied mini-batch normalized SGD with momentum, achieving the in-expectation upper complexity bound of
\begin{align*}
    \fO\left(\left(\frac{L_{0}\Delta}{\epsilon^{2}}+\frac{L_{1}\Delta}{\epsilon}\right)\left(\frac{\sigma^{\frac{p}{p-1}}}{\epsilon^{\frac{p}{p-1}}}+1\right) + \frac{\sigma^{\frac{2p-1}{p-1}}}{\epsilon^{\frac{2p-1}{p-1}}}\right).
\end{align*}

\begin{table*}[t]
\centering
\caption{We summarize the complexity for finding an $\epsilon$-stationary point of nonconvex $(L_0, L_1)$-smooth function, where the oracle complexity is with respect to the exact first-order oracle  and the stochastic first-order oracle for the cases of $\sigma=0$ and $\sigma\geq 0$, respectively. 
} \vskip0.25cm
\label{table:compare}
\resizebox{\linewidth}{!}{\begin{tabular}{ccc}
\hline
Algorithms & Oracle Complexity & Settings \\ 
\hline\addlinespace

\makecell{Clipped GD \\[0.05cm] AGMsDR \\[0.05cm] GD with $\ell$-Smoothness \\ 
\footnotesize \citep{zhang2020gradient,zhang2020improved,koloskova2023revisiting} \\
\footnotesize \citep{vankov2024optimizing,gorbunov2025methods,pmlr-v267-tyurin25a}  
} 
& $\displaystyle{\fO\left(\frac{L_{0}\Delta}{\epsilon^2} + \frac{L_{1}\Delta}{\epsilon}\right)}$ 
& $\sigma=0$ 
\\ \addlinespace

\makecell{GD with Line Search  \\ \footnotesize \citep{hubler2024parameter} } 
& $\displaystyle{\fO\left(
\frac{L_{0}\Delta+L_{1}^{2}\Delta^{2}}{\epsilon^{2}}
\right)}$ 
& $\sigma=0$
\\ \addlinespace

\makecell{$^\dag$Clipped SGD  \\ \footnotesize \citep{zhang2020gradient} } 

&  $\displaystyle{\fO\left(\frac{(L_0^2\sigma/L_1+L_1\sigma^3+L_0 \sigma^{2} +\Delta)^2}{\epsilon^{4}} + \frac{L_{0}\Delta }{\epsilon^{2}}+\frac{L_{1}\Delta}{\epsilon}\right)}$ 

& $\sigma\geq 0$, $p=2$
\\ \addlinespace

\makecell{$^\dag$Clipping  Framework \\ \footnotesize \citep{zhang2020improved} } 
& $\displaystyle{
\fO\left(
\frac{L_{0}\sigma^{2}\Delta }{\epsilon^{4}}
+ \frac{L_{0}\Delta}{\epsilon^{2}}+ 
\frac{L_{1}^{4}\sigma^{2} \Delta}{L_{0}^{3}} +\frac{L_{1}^{2}\Delta}{L_{0}} \right)}$
& $\sigma\geq 0$, $p=2$
\\ \addlinespace

\makecell{DNSGD \\ \footnotesize \citep{luo2026decentralized} } 
& $\displaystyle{
\fO\left(
\left(\frac{L_{0}\Delta}{\epsilon^{2}}+\frac{L_{1}\Delta}{\epsilon}\right)\left(\frac{\sigma^{2}}{\epsilon^{2}}+\frac{L_{1}^{2}\sigma^{2}}{L_{0}^{2}}+1\right)
\right)}$
& $\sigma\geq 0$, $p=2$
\\ \addlinespace

\makecell{Clipped SGD\\ \footnotesize \citep{koloskova2023revisiting} } 
& $\displaystyle{
\fO\left(
\left(\frac{L_{0}\Delta}{\epsilon^{2}}+\frac{L_{1}\Delta}{\epsilon}\left(\frac{\sigma^{2}}{\epsilon^{2}}+1\right)\right)\left(\frac{\sigma^{2}}{\epsilon^{2}}+1\right)
\right)}$
& $\sigma\geq 0$, $p=2$
\\ \addlinespace

\makecell{NSGD-M\\ \footnotesize \citep{hubler2024parameter} } 
& $\displaystyle{
\fO\left(
\frac{(L_{1}\Delta+\sigma+{L_{0}}/{L_{1}})^{4}}{\epsilon^{4}}
\right)}$
& $\sigma\geq 0$, $p=2$
\\ \addlinespace

\makecell{Clipped SGD\\ \footnotesize \citep{chezhegov2025convergence} } 
& $\displaystyle{
\fO\left(
\left(\frac{L_{0}\Delta}{\epsilon^{2}}+\frac{\big(L_{1}\Delta+\sqrt{L_{0}\Delta}\big)\log(1/\delta)}{\epsilon} \right)\left(\left(\frac{\sigma}{\epsilon}\right)^{\frac{p}{p-1}}+1\right)
\right)}$
& $\sigma\geq 0$, $p\in(1,2]$
\\ \addlinespace

\makecell{Batched NSGDM\\ \footnotesize \citep{liu2025nonconvex} } 
& $\displaystyle{
\fO\left(
\left(\frac{L_{0}\Delta}{\epsilon^{2}}+\frac{L_{1}\Delta}{\epsilon}\right)\left(\left(\frac{\sigma}{\epsilon}\right
)^{\frac{p}{p-1}}+1\right)+\left(\frac{\sigma}{\epsilon}\right)^{\frac{2p-1}{p-1}}\right)}$
& $\sigma\geq 0$, $p\in(1,2]$
\\ \addlinespace

\hline\addlinespace

\makecell{(Smoothed) MB-NSGD \\ \footnotesize Theorem \ref{thm:nsgd} } 
& $\displaystyle{
\fO\left(
\left(\frac{L_{0}\Delta}{\epsilon^{2}}+\frac{L_{1}\Delta}{\epsilon}\right)\left(\left(\frac{\sigma}{\epsilon}\right
)^{\frac{p}{p-1}}+1\right)\right)}$
& $\sigma\geq 0$, $p\in(1,2]$
\\ \addlinespace

\makecell{Lower Bound \\ \footnotesize Theorem \ref{thm:L1-lower-bound}} 
&  $\displaystyle{
\Omega\left(
\left(\frac{L_{0}\Delta}{\epsilon^{2}}+\frac{L_{1}\Delta}{\epsilon}\right)\left(\left(\frac{\sigma}{\epsilon}\right
)^{\frac{p}{p-1}}+1\right)\right)}$
& $\sigma\geq 0$, $p\in(1,2]$

\\ \addlinespace

\hline
\end{tabular}} \vspace{-0.2cm}
\begin{enumerate}[leftmargin=0.55cm]
    \item[$^\dag$]\!\!\footnotesize The upper bounds by \citet{zhang2020gradient} and \citet{zhang2020improved} require that there exists some $\sigma\geq 0$ such that for all $\vx\in\BR^d$, the stochastic gradient evaluation $\vg(\vx, Z)$ satisfies $\Norm{\vg(\vx, Z)-\nabla f(\vx)} \leq \sigma$ with probability 1, which is stronger than the bounded variance condition.
\end{enumerate} \vskip-0.25cm
\end{table*}

However, the last term $\sigma^{\frac{2p-1}{p-1}}\epsilon^{-\frac{2p-1}{p-1}}$ does not match the lower bound even in the special case of~$p=2$.
In addition, \citet{chezhegov2025convergence} showed that clipped SGD can achieve the upper complexity bound of
\begin{align*}
\begin{split}    
\fO\left(
\left(\frac{\sigma^{\frac{p}{p-1}}}{\epsilon^{\frac{p}{p-1}}}+1\right)
\left(\frac{L_{0}\Delta}{\epsilon^{2}}+\frac{\big((L_{1}\Delta+\sqrt{L_{0}\Delta}\,\big)\log(1/\delta)}{\epsilon} \right)
\right)
\end{split}
\end{align*}
with probability at least $1-\delta$, which is different from our in-expectation results since it contains an extra logarithmic factor compared to our in-expectation upper bound.

\section{On Hardness of the Individual $(L_0, L_1)$-Smoothness in Finite-Sum Optimization}\label{sec:finite-sum}

In machine learning applications, we usually target to solve the finite-sum optimization problem of the form
\begin{align}\label{eq:prob-finite-sum}
    \min_{\vx\in\BR^d}f(\vx) := \frac{1}{n}\sum_{i=1}^n f_i(\vx).
\end{align}
We focus on the individual $(L_0, L_1)$-smooth setting, i.e., the function class
\begin{align*}
\begin{split}    
&\mathcal{F}_{\mathrm{fs}}(L_0,L_1,\Delta,n) 
= \Big\{(f_1,\dots,f_n): \frac{1}{n}\sum_{i=1}^n f_i(\vzero)- \inf_{\vx\in\BR^d} \frac{1}{n}\sum_{i=1}^nf_i(\vx)\leq\Delta,  \\
&~~~~\Norm{\nabla f_i(\vx)-\nabla f_i(\vy)}  \le  \left( L_0+L_1\Norm{\nabla f_i(\vx)} \right) \Norm{\vx-\vy}~~\text{for all}~~\vx,\vy\in\BR^d~\text{and}~i\in[n] \text{ with }\Norm{\vx-\vy}\le\frac{1}{L_1} \Big\}.
\end{split}
\end{align*}
Note that the assumption that each individual function $f_i(\cdot)$ is $(L_0, L_1)$-smooth cannot guarantee the objective $f(\cdot)$ is also $(L_0, L_1)$-smooth \citep[Proposition 3]{luo2026decentralized}.
The existing convergence guarantees for finding an $\epsilon$-stationary point of problem (\ref{eq:prob-finite-sum}) with $(f_1,\dots,f_n)\in\mathcal{F}_{\mathrm{fs}}(L_0,L_1,\Delta,n)$ require
additional assumptions such as the bounded heterogeneity  \citep{luo2026decentralized,jiang2025decentralized,bo2026provably,rudomanov2026bdescending}.
This naturally leads to the
following question: 
\emph{Can we find an $\epsilon$-stationary point by only using the individual $(L_0, L_1)$-smoothness and bounded optimality gap?}

In the remainder of this section, we give a negative answer by providing a hard instance to show that the condition $(f_1,\dots,f_n)\in\mathcal{F}_{\mathrm{fs}}(L_0,L_1,\Delta,n)$ alone admits no finite, dimension-independent incremental first-order oracle (IFO) upper complexity bound in the worst~case.

We consider the incremental first-order oracle $\mathsf O_{f_{\{f_i\}}}$, which consists of a distribution $\fD_{\rm fs}$ on $[n]$ and a mapping~$\mathsf O_{f_{\{f_i\}}}$ such that for all $\vx=(\vy, z)\in\BR^d\times\BR$ and~$j\in[n]$ it holds
\begin{align*}
\mathsf O_{\{f_i\}}(\vx;j):=\nabla f_j(\vx).        
\end{align*}
We focus on the randomized IFO algorithm class $\fA_{\mathrm {IFO}}$, where each~$\mathsf A\in\fA_{\mathrm {IFO}}$ consists of a distribution $\fP_r$ on a measurable space $\mathcal{R}$ of measurable mappings
$\{\mathsf{A}^{(t)}\}_{t\in\BN}$ such that each~$\mathsf{A}^{(t)}$ takes in the first $t-1$ responses and the random seed $r\in\BR$ to produce the $t$th query.
We let~$\big\{\vx_{\mathsf{A}[{\mathsf O}_{\{f_i\}}]}^{(t)}\big\}_{t\in\BN}$
denote the random sequence of queries resulting from the algorithm
$\mathsf{A}$ with oracle~${\mathsf O}_{\{f_i\}}$, defined recursively as
\begin{equation*}
\begin{split}    
\vx_{\mathsf{A}[{\mathsf O}_{\{f_i\}}]}^{(t)} 
= \mathsf{A}^{(t)}
    \Big( 
        r,\, 
        {\mathsf O}_{\{f_i\}}
        \big(
            \vx_{\mathsf{A}[{\mathsf O}_{\{f_i\}}]}^{(1)};
            j^{(1)}
        \big),
        \ldots,  {\mathsf O}_{\{f_i\}}
        \big(
            \vx_{\mathsf{A}[{\mathsf O}_{\{f_i\}}]}^{(t-1)};
            j^{(t-1)}
        \big)
    \Big).
\end{split}
\end{equation*}

For given $L_0, L_1,L,\Delta>0$ and $n\geq 2$, we define the function set $\{\hat f_i:\BR^d\times\BR\to\BR\}_{i=1}^n$ as 
\begin{align}\label{eq:f_i}
\!\!\!\!
\hat f_i(\vy,z)=
\begin{cases}
h(\vy)+\dfrac{Lz}{L_1},
&\!\!\text{if}~i\in[n-1],\!\!\\[3mm]
h(\vy)-\dfrac{(n-1)Lz}{L_1},
&\!\!\text{if}~i=n,
\end{cases}    
\end{align}

where $h:\BR^d\to\BR$ satisfying $h\in\fF(L,0,\Delta)$.
The above construction has the following properties.
\begin{lemma}\label{lem:f_i}
The function $\hat f_i:\BR^{d}\times\BR\to\BR$ defined in \eqref{eq:f_i} is $(L_0, L_1)$-smooth and it holds 
\begin{align*}
\hat f:=\frac{1}{n}\sum_{i=1}^n \hat f_i=h\in\fF(L,0,\Delta).    
\end{align*}
\end{lemma}

We introduce the randomized algorithm class $\fB_{\mathrm{rand}}$, where each $\mathsf B\in\fB_{\mathrm{rand}}$ consists of a distribution $\fP_r$ on a measurable space $\mathcal{R}$ of measurable mappings $\{\mathsf{B}^{(t)}\}_{t\in\BN}$ such that each~$\mathsf{B}^{(t)}$ takes in the first $t-1$ responses and the random seed $r\in\BR$ to produce the $t$th query.
We let~$\big\{\vy_{\mathsf{B}[{\mathsf O}_h]}^{(t)}\big\}_{t\in\BN}$
denote the random sequence of queries resulting from the algorithm
$\mathsf{B}\in\fB_{\mathrm{rand}}$ with oracle~${\mathsf O}_h(\vx):=\nabla h(\vx)$, defined recursively as
\begin{equation*}
\begin{split}    
\vy_{\mathsf{B}[{\mathsf O}_{h}]}^{(t)} 
= \mathsf{B}^{(t)}
    \Big( &
        r,\,
        {\mathsf O}_{h}
        \big(
            \vy_{\mathsf{B}[{\mathsf O}_{h}]}^{(1)}
        \big),
        \ldots, {\mathsf O}_{h}
        \big(
            \vy_{\mathsf{B}[{\mathsf O}_{h}]}^{(t-1)}
        \big)
    \Big).
\end{split}
\end{equation*}
We denote
\begin{align*}
\begin{split}    
\mathcal T_\epsilon(\mathsf B,h)
\!:=\!\inf\left\{
T\in\mathbb N:
\mathbb P\left(
\min_{t\in[T]}
\left\|\nabla h\big(\vy_{{\mathsf B}[{\mathsf O}_h ]}^{(t)}\big)\right\|
\le\epsilon
\right)\ge\frac12
\right\}.
\end{split}
\end{align*}
In the view of finding an $\epsilon$-stationary point of the function $h\in\fF(L,0,\Delta)$ by using randomized algorithm $\mathsf B\in\fB_{\mathrm{rand}}$ with first-order oracle ${\mathsf O}_h$, Theorem 2 of \citet{carmon2020lower} with $p=1$ indicates there exists numerical constant $c>0$ such that for any $L,\Delta,\epsilon>0$, it holds 
\begin{align}\label{eq:TBh}
  \inf_{{\mathsf B}\in\mathcal B_{\mathrm{rand}}} \sup_{h\in\mathcal F(L,0,\Delta)}  
\mathcal T_\epsilon(\mathsf B,h) \geq \frac{cL\Delta}{\epsilon^2},
\end{align}
where the supremum over $h\in\mathcal F(L,0,\Delta)$ ranges over all finite dimensions.

For the randomized IFO algorithm class ${\mathcal A}_{\rm IFO}$, we denote
\begin{align*}
 \mathcal T_\epsilon^{\mathrm{IFO}}(\mathsf A,  (f_1,\dots, f_n)) 
 :=
\inf\left\{
T\in\mathbb N:
\mathbb P\left(
\min_{t\in[T]}
\left\|\nabla f\bigl(\vx_{\mathsf A[\mathsf O_{\{f_i\}}]}^{(t)}\bigr)\right\|
\le\epsilon
\right)\ge\frac12
\right\}.
\end{align*}
Note that definition (\ref{eq:f_i}) and Lemma \ref{lem:f_i} imply that for every fixed $h$ and the random seed, it holds
\begin{align}\label{eq:f-h-same1}
\nabla_{\vy} \hat f_i\bigl(\vx_{\mathsf A[\mathsf O_{\{\hat f_i\}}]}^{(t)}\bigr)
=\nabla h\bigl(\vy_{\mathsf B[\mathsf O_h]}^{(t)}\bigr)
\end{align}
and
\begin{align}\label{eq:f-h-same2}
\Norm{\nabla \hat f\bigl(\vx_{\mathsf A[\mathsf O_{\{\hat f_i\}}]}^{(t)}\bigr)}
=\Norm{\nabla h\bigl(\vy_{\mathsf B[\mathsf O_h]}^{(t)}\bigr)}
\end{align}
for all $t\in\BN$.
According to equations (\ref{eq:f-h-same1}), (\ref{eq:f-h-same2}), and  (\ref{eq:TBh}), for any $L_0, L_1, L, \Delta,\epsilon>0$, and $n\geq 2$, it holds
\begin{align*}    
&\inf_{\mathsf A\in\mathcal A_{\mathrm{IFO}}}\,
\sup_{(f_1,\dots,f_n)\in\mathcal F_{\mathrm{fs}}(L_0,L_1,\Delta,n) }
\mathcal T_{2\epsilon}^{\mathrm{IFO}}(\mathsf A,(f_1,\dots,f_n))
\\
&\quad\ge
\inf_{\mathsf A\in\mathcal A_{\mathrm{IFO}}}\,
\sup_{h\in\mathcal F(L,0,\Delta)}
\mathcal T_{2\epsilon}^{\mathrm{IFO}}(\mathsf A, (\hat f_1,\dots,\hat f_n))
\ge
\inf_{\mathsf B\in\fB_{\mathrm{rand}}}\,
\sup_{h\in\mathcal F(L,0,\Delta)}
\mathcal T_{2\epsilon}(\mathsf B,h)
 \ge
\frac{cL\Delta}{4\epsilon^2}.
\end{align*}
We emphasize that the above lower bound does not depend on the values of $L_0$ and $L_1$.
Hence, for any $L_0, L_1, \Delta,\epsilon>0$,  $\mathsf A\in\mathcal A_{\mathrm{IFO}}$, and \mbox{$T\in\BN$}, we can take $\hat f_1,\dots,\hat f_n$ by following \eqref{eq:f_i} with $L=4(T+1)\epsilon^2/(c\Delta)$
to achieve
\begin{align*}
\mathcal T_{2\epsilon}^{\mathrm{IFO}}(\mathsf A,(\hat f_1,\dots,\hat f_n)) \ge \frac{cL\Delta}{4\epsilon^2} > T,
\end{align*}
which implies
\[
\mathbb P\left(
\min_{t\in[T]}
\left\|
\nabla \hat f\bigl(
\vx_{\mathsf A[\mathsf O_{\{f_i\}}]}^{(t)}
\bigr)
\right\|>2\epsilon
\right)>\frac12.
\]
Consequently, we have
\begin{align*}
\min_{t\in[T]}
\mathbb E\left\|
\nabla f\bigl(
\vx_{\mathsf A[\mathsf O_{\{f_i\}}]}^{(t)}
\bigr)
\right\| 
 \ge
\mathbb E\left[
\min_{t\in[T]}
\left\|
\nabla f\bigl(
\vx_{\mathsf A[\mathsf O_{\{f_i\}}]}^{(t)}
\bigr)
\right\|
\right]\ge
2\epsilon\,
\mathbb P\left(
\min_{t\in[T]}
\left\|
\nabla f\bigl(
\vx_{\mathsf A[\mathsf O_{\{f_i\}}]}^{(t)}
\bigr)
\right\|>2\epsilon
\right) >\epsilon.
\end{align*}
Therefore, for any fixed $n\ge2$, $L_0,L_1,\Delta,\epsilon>0$, there is no finite, dimension-independent IFO upper complexity bound for finding an $\epsilon$-stationary point of problem (\ref{eq:prob-finite-sum}) with $(f_1,\dots,f_n)\in\mathcal{F}_{\mathrm{fs}}(L_0,L_1,\Delta,n)$ in the worst~case.
                
\section{Conclusion}\label{sec:conclusion}

This paper establishes an SFO lower complexity bound of $\Omega\big((\sigma^{\frac{p}{p-1}}\epsilon^{-\frac{p}{p-1}}+1)
(L_0\Delta\epsilon^{-2}+L_1\Delta\epsilon^{-1})\big)$ for finding $\epsilon$-stationary points of $(L_0, L_1)$-smooth functions.
We then analyze the convergence of (smoothed) mini-batch normalized SGD to achieve a matching upper complexity bound.
Moreover, we verify the hardness of finite-sum nonconvex optimization under individual $(L_0, L_1)$-smoothness.
An interesting direction for future work is to extend our results to the distributed setting \citep{luo2026decentralized,jiang2025decentralized,bo2026provably,bo2026linear} to establish lower bounds on the communication complexity.

\section*{AI Disclosure}
The construction in the lower‑bound was discovered with the assistance of ChatGPT-5.5 Sol and ChatGPT-6 Astra.  
The authors independently verified the arguments, and assume responsibility for all content.
 
\bibliographystyle{plainnat}
\bibliography{reference}

\appendix

\section{The Proofs for Results in Section \ref{sec:L0-L1-smooth}}
This section provides the proofs of our construction for $(L_0, L_1)$-smooth function.
\subsection{The Proof of Lemma~\ref{lem:gate-properties}}
\begin{proof}
Recall that the functions $\psi$ and $\phi$ are defined as
\begin{equation}
    \psi(t):=S(2t-1) \qquad \text{and}   \qquad   \phi(t):=1+\tanh(t).
    \label{eq:gate-functions1}
\end{equation}
respectively, where
\begin{equation}
    S(u)=
    \begin{cases}
        0, & u\le 0,\\
        10u^3-15u^4+6u^5, & u\in(0,1),\\
        1, & u\ge 1.
    \end{cases}
    \label{eq:gate-S1}
\end{equation}
We verify the first-order and second-order derivatives of $S$ as
\begin{align}\label{eq:S-prime}
S'(u) = \begin{cases}
30u^2(u-1)^2,  & \text{if}~u\in(0,1), \\
0, & \text{otherwise},
\end{cases}
\end{align}
and
\begin{align}\label{eq:S-second}
    S''(u) =\begin{cases}
    60u(u-1)(2u-1),  & \text{if}~u\in(0,1), \\
    0, & \text{otherwise}.
    \end{cases}
\end{align}
We now prove each statement of the lemma.

\paragraph{The proof of statement (a).} The definitions of $\psi$ and $S$ in equations (\ref{eq:gate-S1}) and (\ref{eq:S-second}) directly show that~$\psi(t)=0$ for all $t \le 1/2$, and $\psi(t)=1$ for all $t\ge1$.

\paragraph{The proof of statement (b).}
We first consider the value of $\psi(t)$. 
By equation (\ref{eq:S-prime}), we know that $S'(t)\geq 0$ on $\BR$, which implies the function $S$ is non-decreasing.
Combining the facts $S(0)=0$ and $S(1)=1$, we obtain
$0\le S(u)\le1$ for all $u\in[0,1]$.
Combining the definitions $S(u)=0$ for all~$u\leq 0$ and $S(u)=1$ for all $u\geq 1$, we know that $S(u)\in[0,1]$ on $\BR$.
Hence, the definition of~$\psi(t)=S(2t-1)$ implies
\begin{equation*}%\label{eq:psi-num}
    0\le\psi(t)\le1.
\end{equation*}

We then consider the value of $\psi'(t)$. By the chain rule, we can write 
\begin{align}\label{eq:psi-dev}
\psi'(t) = \begin{cases}
 60(2t-1)^2(2t-2)^2,       &  t\in[ 1/2,1], \\
 0, & \text{otherwise}.
\end{cases}
\end{align}
Clearly, it holds $\psi'(t)\ge0$. 
For $t\in[1/2,1]$, we let
\[
u:=2t-1\in[0,1],
\]
which leads to
\begin{align*}
\psi'(t)= 60(2t-1)^2(2t-2)^2 = 60u^2(u-1)^2 .
\end{align*}
Since for all $u\in[0,1]$, it holds 
\[
    u(u-1) = \left(u-\frac12\right)^2 - \frac14
    \in \left[-\frac14, 0\right],
\]
which implies
\[
    \psi'(t)
    \le
    60\cdot\frac1{16}
    =
    \frac{15}{4}
    <4.
\]
Recall that $\psi'(t)=0$ for all $t\in[0, 1/2]$, 
we obtain $0\le\psi'(t)\le4$ for all $t\in\BR$.

We now consider the value of $\psi''(t)$. By the chain rule, we can write 
\[
    \psi''(t)=4S''(2t-1).
\]
Equation (\ref{eq:S-second}) indicates $\psi''(t)=0$ for all $t\not\in(1/2,1)$.
Hence, we focus on the case of~$t\in(1/2,1)$, that is $u=2t-1\in(0,1)$. 
We write
\begin{equation*}
    \psi''(t)=240u(1-u)(1-2u),
    %\label{eq:psi-second-transition}
\end{equation*}
and define
\[
    g(u):=u(1-u)(1-2u)
    =
    u-3u^2+2u^3.
\]
Then we have
\[
    g'(u)=1-6u+6u^2,
\]
which has critical points 
\[
    u_{+} = \frac{3+\sqrt3}{6}
    \qquad\text{and}\qquad
    u_{-} = \frac{3-\sqrt3}{6}.
\]
Since it holds $g(0)=g(1)=0$, the maximum of $|g(u)|$ on $[0,1]$ is attained at one of these critical points.
We can verify
\begin{align*}
    |g(u_+)| = |g(u_-)|=\frac{\sqrt3}{18},
\end{align*}
which implies
\[
    |\psi''(t)|
    \le
    240 \cdot \frac{\sqrt3}{18}
    <24.
\]

\paragraph{The proof of statement  (c).}
The fact $\tanh t\in(-1,1)$ directly implies 
\begin{align*}
\phi(t)=1+\tanh t\in(0,2).    
\end{align*}
We can also verify
$\phi'(t)=\operatorname{sech}^2t$ and $\phi''(t) = -2\operatorname{sech}^2(t)\tanh(t)$, which implies
\begin{align*}
    |\phi''(t)|
= 2\operatorname{sech}^2(t)|\tanh(t)|
< 2\operatorname{sech}^2(t)
= 2\phi'(t).
\end{align*}

\paragraph{The proof of statement  (d).}
The statement that $\phi'$ is even can be verified as 
\begin{equation*}%\label{eq:phi-derivative-even}
    \phi'(-t)=\operatorname{sech}^2(-t)=\operatorname{sech}^2(t)=\phi'(t).
\end{equation*}
It remains to show that $\phi'(t)>2/5$ for $|t|\le1$.
For all $t\geq 0$, it holds
\[
    \phi''(t) =
    -2\operatorname{sech}^2(t)\,\tanh (t)    \le0,
\]
which implies $\phi'$ is non-increasing on $[0,\infty)$.
Therefore, for all $t\in[0,1]$, it holds
\begin{equation*}
    \phi'(t) =\operatorname{sech}^2(t)
    \ge \operatorname{sech}^2(1)
    > \frac{2}{5}.
    %\label{eq:phi-local-minimum}
\end{equation*}
Since $\phi'(t)$ is even, we have $|\phi'(t)|>2/5$ when $|t|\leq 1$.

\end{proof}

\subsection{Additional Properties of the Component Function}

We additionally establish the following useful properties for function $\psi$, which will be used in the subsequent analysis.

\begin{lemma}\label{lem:gate-paired-properties}
For every $t\in\BR$, the function $\psi$ defined in \eqref{eq:gate-functions} satisfies
\begin{align}
    \psi(t)+\psi(-t)
    &\le1,
    \label{eq:gate-pair-value}
    \\
    \psi'(t)+\psi'(-t)
    &\le4,
    \label{eq:gate-pair-first}
    \\
    |\psi''(t)|+|\psi''(-t)|
    &\le24.
    \label{eq:gate-pair-second}
\end{align}
\end{lemma}
\begin{proof}
The expression of $\psi'(\cdot)$ in \eqref{eq:psi-dev} implies
\begin{equation}
    \psi'(t)=\psi''(t)=0
    \qquad
    \text{whenever } t\leq 1/2  \text{~~or~~}  t \geq 1.
    \label{eq:gate-pair-derivative-zero}
\end{equation}
We prove this lemma by considering the following three cases.

\paragraph{The Case of $|t|\le1/2$.}
The condition $|t|\le1/2$ implies both $t$ and $-t$ are at most $1/2$.
Hence, the definitions of $\psi(\cdot)$ and $S(\cdot)$ in equations (\ref{eq:gate-functions1}) and (\ref{eq:gate-S1}) imply $\psi(t)=\psi(-t)=0$. This results
\[
    \psi(t)+\psi(-t)=0\le1.
\]
Furthermore, \eqref{eq:gate-pair-derivative-zero} indicates
$\psi'(t)=\psi'(-t)=0$ and $\psi''(t)=\psi''(-t)=0$, which imply
\[
\psi'(t)+\psi'(-t)=0\le4
\qquad\text{and}\qquad
|\psi''(t)|+|\psi''(-t)|=0\le24.
\]

\paragraph{The Case of $t>1/2$.}
The condition $t>1/2$ implies $-t<-1/2$. Hence, the definition (\ref{eq:gate-functions1})  and \eqref{eq:gate-pair-derivative-zero} imply
\begin{align*}
\psi(-t)=\psi'(-t)=\psi''(-t)=0.     
\end{align*}
Combining with Lemma~\ref{lem:gate-properties}(b), we obtain
\[
    \psi(t)+\psi(-t)=\psi(t)\le1,~~
    \psi'(t)+\psi'(-t)=\psi'(t)\le4,~~\text{and}~~
    |\psi''(t)|+|\psi''(-t)|=|\psi''(t)|\le24.
\]
\paragraph{The Case of $t<-1/2$.}
The condition $t<-1/2$ implies $-t>1/2$. Hence, the definition  (\ref{eq:gate-functions1}) and \eqref{eq:gate-pair-derivative-zero} imply
\begin{align*}
    \psi(t)=\psi'(t)=\psi''(t)=0.
\end{align*}
Combining with Lemma~\ref{lem:gate-properties}(b), we obtain
\[
\psi(t)+\psi(-t)=\psi(-t)\le1,~~~\psi'(t)+\psi'(-t)=\psi'(-t)\le4,~~~\text{and}~~~
|\psi''(t)|+|\psi''(-t)|=|\psi''(-t)|\le24.
\]

Hence, equations (\ref{eq:gate-pair-value}), (\ref{eq:gate-pair-first}), and (\ref{eq:gate-pair-second}) hold for all $t\in\BR$.
\end{proof}

\subsection{The Proof of Lemma~\ref{lem:hard-chain-properties}}
\begin{proof}
Recall that the function $\bar f_T$ is defined as
\begin{equation}
\bar f_T(\vy) := \sum_{i=1}^{T}A_i(\vy),
~~~~\text{where}~~
A_i(\vy):=
\begin{cases}
-\phi(y_1), & \text{if}~i=1, \\[0.1cm]
\psi(-y_{i-1})\phi(-y_i)
-
\psi(y_{i-1})\phi(y_i), & \text{otherwise}.    
\end{cases}
\label{eq:hard-zero-chain1}
\end{equation}

\paragraph{The proof of statement (a).}
According to  the definition (\ref{eq:gate-functions1}) and Lemma~\ref{lem:gate-properties}(a), we have $\phi(0)=1$ and~$\psi(0)=0$.
Hence, definition (\ref{eq:hard-zero-chain1}) implies
\begin{equation}
    A_i(\vzero)=\begin{cases}
        -1, &  \text{if}~i=1, \\
        0, &  \text{otherwise},
    \end{cases}
    \qquad\text{and}\qquad     
    \bar f_T(\vzero)=-1.
    \label{eq:hard-chain-origin}
\end{equation}
For all $\vy\in\BR^T$, Lemma~\ref{lem:gate-properties}(c) implies
\[
    A_1(\vy)=-\phi(y_1)>-2.
\]
Moreover, for every $i\ge2$, Lemma~\ref{lem:gate-properties}(b) and (c) implies
\[
\psi(-y_{i-1})\phi(-y_i)\ge0
\qquad\text{and}\qquad
0\le\psi(y_{i-1})\phi(y_i)<2.
\]
It follows from \eqref{eq:hard-zero-chain1}
that
\[
    A_i(\vy)>-2
    \qquad\text{and}\qquad
    \bar f_T(\vy) = \sum_{i=1}^T A_i(\vy)>-2T.
\]
Thus, it holds
\[
    \inf_{\vy\in\BR^T}\bar f_T(\vy)\ge-2T.
\]
We then consider the vector
\[
    \vy^{(r)}= r[1,\ldots,1]^\top,
    \qquad \text{where}~~r\ge1.
\]
Lemma~\ref{lem:gate-properties}(a) implies $\psi(r)=1$ and $\psi(-r)=0$.
Consequently, definition \eqref{eq:hard-zero-chain1} implies
\[
    A_i(\vy^{(r)})
    =
    -\phi(r)
\]
for all $i\in[T]$, which results
\begin{equation}
    \lim_{r\to+\infty} \bar f_T(\vy^{(r)}) = \lim_{r\to+\infty} -T\phi(r) = -2T.
    \label{eq:hard-chain-infimum}
\end{equation}
Combining equations (\ref{eq:hard-chain-origin})
and (\ref{eq:hard-chain-infimum}) yields
\begin{equation*}
    \bar f_T(\vzero)
    -
    \inf_{\vy\in\BR^T}\bar f_T(\vy)
    =
    2T-1.
\end{equation*}

\paragraph{The proof of statement (b).}
We first establish the uniform bound on the gradient of $\bar f_T$. Taking the derivative on \eqref{eq:hard-zero-chain1} and using the evenness of $\phi'$ from Lemma~\ref{lem:gate-properties}(d), we obtain
\begin{equation}\label{eq:hard-chain-gradient-coordinate}
\begin{split}    
\frac{\partial\bar f_T(\vy)}{\partial y_i} = 
\begin{cases}
    -\phi'(y_1) -\psi'(-y_1)\phi(-y_2) -\psi'(y_1)\phi(y_2), & i = 1, \\
    -\big(\psi(-y_{i-1})+\psi(y_{i-1})\big)\phi'(y_i) -\psi'(-y_i)\phi(-y_{i+1})-\psi'(y_i)\phi(y_{i+1}), & i=2,\dots,T-1, \\
    -\big(\psi(-y_{T-1})+\psi(y_{T-1})\big)\phi'(y_T), & i=T. 
\end{cases}
\end{split}
\end{equation}
According to Lemma~\ref{lem:gate-properties}(b) and (c) and Lemma~\ref{lem:gate-paired-properties}, we have
\[
    0\le (\psi(-y_{i-1})+\psi(y_{i-1}))\phi'(y_i) \le1
~~~~~~\text{and}~~~~~~
    0\le
    \psi'(-y_i)\phi(-y_{i+1})
    +\psi'(y_i)\phi(y_{i+1})
    \le8.
\]
Combining with \eqref{eq:hard-chain-gradient-coordinate} implies
\begin{equation*}%\label{eq:hard-chain-gradient-bound}
    \Norm{\nabla\bar f_T(\vy)} 
    \le  \sqrt{T} \max_{i\in[T]}\left|\frac{\partial\bar f_T}{\partial y_i}\right|
\le9\sqrt T.
\end{equation*}
We now consider the bound of Hessian. 
Lemma~\ref{lem:gate-properties}(c) implies
\[
    \left|\frac{\partial^2 A_1(\vy)}{\partial y_1^2}\right|
    =|\phi''(y_1)|\le2.
\]
For all $i\ge2$, Lemmas~\ref{lem:gate-properties}(c) and \ref{lem:gate-paired-properties} imply
\begin{align}
& \left|\frac{\partial^2 A_i(\vy)}{\partial y_{i-1}^2}\right|
= |\psi''(-y_{i-1})\phi(-y_i)     -\psi''(y_{i-1})\phi(y_i)|
\le2(|\psi''(-y_{i-1})|+|\psi''(y_{i-1})|)
\le48, \label{eq:hard-A-xx-bound} \\[0.15cm] 
& \left|\frac{\partial^2 A_i(\vy)}{\partial y_i^2}\right|
= |-(\psi(-y_{i-1})+\psi(y_{i-1}))\phi''(y_i)| \le2, \label{eq:hard-A-zz-bound} \\[0.15cm]
& \left|\frac{\partial^2 A_i(\vy)}{\partial y_{i-1}\partial y_i}\right|
 = |(\psi'(-y_{i-1})-\psi'(y_{i-1}))\phi'(y_i)|
\le4.    \label{eq:hard-A-xz-bound}
\end{align}
Note that each $A_i(\vy)$ depends only on $y_{i-1}$ and $y_i$ for $i\ge2$, which implies the Hessian of $\bar f_T(\vy)$ is tridiagonal matrix, i.e.,
\begin{equation*}
    \frac{\partial^2\bar f_T(\vy)}{\partial y_i\partial y_j}=0
    ~~~\text{for all}~~~i,j\in[T]~~~\text{such that}~~~|i-j|\ge2.
    %\label{eq:hard-chain-hessian-sparsity}
\end{equation*}
According to equations (\ref{eq:hard-A-xx-bound}) and (\ref{eq:hard-A-zz-bound}), the absolute value of each diagonal entry of $\nabla^2\bar f_T(\vy)$ is bounded by $48+2=50$. 
Moreover, \eqref{eq:hard-A-xz-bound} implies that each nonzero off-diagonal entry of $\nabla^2\bar f_T(\vy)$ has absolute value at most $4$.
Hence, every row of $\nabla^2\bar f_T(\vy)$ has absolute row sum at most $50+2\cdot4=58$. 
Since the Hessian is symmetric, its spectral norm is bounded by its maximum absolute row sum, which implies
\begin{equation}
    \Norm{\nabla^2\bar f_T(\vy)}
    \le\Norm{\nabla^2\bar f_T(\vy)}_\infty
    \le58.
    \label{eq:hard-chain-hessian-bound}
\end{equation}

\paragraph{The proof of statement (c).}
We let $i\in[T]$ be the smallest index such that $|y_i|\le1$. 
If $i=1$, then \eqref{eq:hard-chain-gradient-coordinate} and the nonnegativity of $\psi'$ and $\phi$ from Lemma~\ref{lem:gate-properties}(b) and (c) implies
\[
    \frac{\partial\bar f_T(\vy)}{\partial y_1}
    \le-\phi'(y_1).
\]
Combining with $|y_1|\le1$ and Lemma~\ref{lem:gate-properties}(d), we achieve
\[
    \left|\frac{\partial\bar f_T(\vy)}{\partial y_1}\right|
    \ge\phi'(y_1)>\frac25.
\]
If $i\ge2$, the definition of $i$ implies $|y_{i-1}|>1$. 
According to Lemma~\ref{lem:gate-properties}(a), we obtain
\begin{equation}
    \psi(-y_{i-1})+\psi(y_{i-1})=1.
    \label{eq:previous-gate-one}
\end{equation}
Substituting \eqref{eq:previous-gate-one} into \eqref{eq:hard-chain-gradient-coordinate} and using the non-negativity of $\psi'$ and $\phi$ from Lemma~\ref{lem:gate-properties}(b) and (c), we have
\[
    \frac{\partial\bar f_T(\vy)}{\partial y_i}
    \le-\phi'(y_i).
\]
Since $|y_i|\le1$, Lemma~\ref{lem:gate-properties}(d) implies
\[
    \left|\frac{\partial\bar f_T(\vy)}{\partial y_i}\right|
    \ge\phi'(y_i)>\frac25.
\]
Combining the above two inequalities, we obtain
\begin{equation}
    \Norm{\nabla\bar f_T(\vy)}
    \ge
    \left|\frac{\partial\bar f_T}{\partial y_i}\right|
    >
    \frac25.
    \label{eq:hard-chain-certificate}
\end{equation}

\paragraph{The proof of statement (d).}
We now establish the gradient-dependent upper bound for Hessian.
We first consider the case that there exists some $i\in[T]$ such that $|y_i|\le1$. 
By equations (\ref{eq:hard-chain-hessian-bound}) and (\ref{eq:hard-chain-certificate}), we obtain
\begin{equation}
    \Norm{\nabla^2\bar f_T(\vy)}
    \le 58
    \leq 160\cdot\frac25
    \le160\Norm{\nabla\bar f_T(\vy)}.
    \label{eq:hard-chain-curvature-case-one}
\end{equation}
We then consider the case that $|y_i|>1$ holds for all $i\in[T]$. 
According to Lemma~\ref{lem:gate-properties}(a), we have
\[
    \psi(-y_i)+\psi(y_i)=1.
\]
Moreover, since both $y_i$ and $-y_i$ do not lie in the interval $(1/2,1)$, \eqref{eq:gate-pair-derivative-zero} implies
\[
    \psi'(-y_i)=\psi'(y_i)=0.
\]
Substituting the above two identities into \eqref{eq:hard-chain-gradient-coordinate}, we obtain
\begin{equation}
    \frac{\partial\bar f_T(\vy)}{\partial y_i}
    =-\phi'(y_i)
    \quad \text{for all}~~i\in[T].
    \label{eq:hard-chain-saturated-gradient}
\end{equation}
Hence, the Hessian of $\bar f_T(\vy)$ is diagonal and satisfies
\begin{equation*}
    \nabla^2\bar f_T(\vy)=-\operatorname{diag}\big(\phi''(y_1),\ldots,\phi''(y_T)\big).
    %\label{eq:hard-chain-saturated-hessian}
\end{equation*}
Consequently, Lemma~\ref{lem:gate-properties}(c) and \eqref{eq:hard-chain-saturated-gradient} imply
\begin{align}
    \Norm{\nabla^2\bar f_T(\vy)}=
    \max_{i\in[T]}|\phi''(y_i)|
    \le2\max_{i\in[T]}\phi'(y_i)
    \le2\Norm{\nabla\bar f_T(\vy)}
    \le160\Norm{\nabla\bar f_T(\vy)}.
    \label{eq:hard-chain-curvature-case-two}
\end{align}
Combining equations (\ref{eq:hard-chain-curvature-case-one}) and  (\ref{eq:hard-chain-curvature-case-two}), we conclude that
\begin{equation*}
    \Norm{\nabla^2\bar f_T(\vy)}
    \le160\Norm{\nabla\bar f_T(\vy)}
    \qquad
    %\label{eq:hard-chain-hessian-gradient}
\end{equation*}
for all $\vy\in\BR^T$.

\paragraph{The proof of statement (e).}
Follow the notation defined in Section \ref{sec:Preliminaries}, we let
\begin{align*}
    \mathrm{prog}_{1/4}(\vy):=\max\{i\in\BN : |y_i| > 1/4\}
\end{align*}
for a given $\vy=[y_1, \dots, y_T]\in\BR^T$.
For each $i>\operatorname{prog}_{1/4}(\vy)+1$, we have 
\begin{align}\label{eq:y_i14}
|y_i| \le 1/4 \qquad\text{and}\qquad |y_{i-1}| \le 1/4.    
\end{align}
Combining Lemma~\ref{lem:gate-properties}(a), equations (\ref{eq:gate-pair-derivative-zero}) and (\ref{eq:y_i14}) implies
\[
    \psi(-y_{i-1})=\psi(y_{i-1})=0  \qquad\text{and}\qquad     \psi'(-y_i)=\psi'(y_i)=0.
\]
Substituting the above results into \eqref{eq:hard-chain-gradient-coordinate}, we obtain
\[
    \frac{\partial\bar f_T(\vy)}{\partial y_i}=0
\]
for all $i>\operatorname{prog}_{1/4}(\vy)+1$.
Hence, it holds
\begin{equation*}
    \operatorname{prog}_0\big(\nabla\bar f_T(\vy)\big)
    \le\operatorname{prog}_{1/4}(\vy)+1.
    %\label{eq:hard-chain-progress-proof}
\end{equation*}
\end{proof}

\subsection{The Proof of Lemma~\ref{lem:global-hard-properties}}
Before proving Lemma~\ref{lem:global-hard-properties}, we first establish the following properties of the mapping 
\begin{equation}
    \rho_R(\vx)
    :=
    \frac{\vx}{\sqrt{1+\Norm{\vx}^2/R^2}}.
    \label{eq:radial-compression1}
\end{equation}

\begin{lemma}
For all $\vx\in\BR^d$, the mapping $\rho_R:\BR^d\to\BR^d$ defined in \eqref{eq:radial-compression} holds
\begin{align*}
    \Norm{\rho_R(\vx)}\le R, \qquad 
    \Norm{D\rho_R(\vx)}\le1, \qquad \text{and} \qquad 
    \Norm{D^2\rho_R(\vx)}\le\frac3R.
\end{align*}
\end{lemma}

\begin{proof}
Following \citet[Lemma 13 and its proof]{arjevani2023lower}, we directly obtain
\begin{equation}
    \Norm{\rho_R(\vx)}\le R
    %\label{eq:radial-image-bound}
    \qquad\text{and}\qquad
    \Norm{D\rho_R(\vx)}\le1.
    \label{eq:radial-jacobian-bound}
\end{equation}
Here, we focus on the upper bound of $\Norm{D^2\rho_R(\vx)}$. We denote
\begin{equation}
    \alpha(\vx):=
    \left(1+\frac{\Norm{\vx}^2}{R^2}\right)^{-1/2}.
    \label{eq:radial-alpha}
\end{equation}
Then we can write $\rho_R(\vx)=\alpha(\vx)\vx$ and it holds  \citep[Lemma~13]{arjevani2023lower} 
\begin{equation}
    D\rho_R(\vx)
    = \alpha(\vx)\mI_d - \frac{(\alpha(\vx))^3}{R^2}\cdot\vx\vx^\top.
    \label{eq:radial-jacobian}
\end{equation}
According to \citet[Lemma~13]{arjevani2023lower}, it holds
\begin{equation}
    \Norm{D\rho_R(\vx)-D\rho_R(\vy)}
    \le\frac3R\Norm{\vx-\vy}
    \qquad\text{for all }\vx,\vy\in\BR^d.
    \label{eq:radial-jacobian-lipschitz}
\end{equation}
Since the mapping $\rho_R$ is twice continuously differentiable, for any $\vh,\vk\in\BR^d$, we have
\begin{align*}
    \Norm{D^2\rho_R(\vx)[\vh,\vk]}
    =
    \lim_{t\to0}
    \left\|
    \frac{
    \big(D\rho_R(\vx+t\vk)-D\rho_R(\vx)\big)\vh
    }{t}
    \right\|\le
    \frac3R\Norm{\vh}\Norm{\vk},
    %\label{eq:radial-second-directional-bound}
\end{align*}
where the inequality follows from \eqref{eq:radial-jacobian-lipschitz}. Taking the supremum over all unit vectors $\vh$ and $\vk$, we achieve
\begin{equation}
    \Norm{D^2\rho_R(\vx)}
    \le\frac3R.
    \label{eq:radial-second-bound}
\end{equation}
This completes the proof.
\end{proof}

\subsection{The Proof of Lemma~\ref{lem:global-hard-properties}}
\begin{proof}
Recall that \eqref{eq:global-hard-function} has defined
\[
    H_{\mU,\mu}(\vx)
    =
    \bar f_T\big(\mU^\top\rho_R(\vx)\big)
    +\frac{\mu}{2}\Norm{\vx}^2
    \quad \text{with} \quad
    R=\frac{1000\sqrt T}{\mu}
    \quad\text{and}\quad
    \mu\in\left(0,\frac{1}{100}\right],
\]
where $\mU=[\vu_1,\ldots,\vu_T]\in\operatorname{St}(d,T)$. 
Throughout the proof of this lemma, we denote
\begin{align*}
\vy:=\mU^\top\rho_R(\vx),    
\end{align*}
where $\rho_R(\vx)=\alpha(\vx)\vx$ follows \eqref{eq:radial-alpha}. 
Additionally, Lemma~\ref{lem:hard-chain-properties}(b) implies
\begin{equation}
    \Norm{\nabla\bar f_T(\vy)}
    \le9\sqrt T % \le10\sqrt T,
    \qquad \text{and} \qquad
    \Norm{\nabla^2\bar f_T(\vy)}\le 58.
    \label{eq:global-chain-bounds}
\end{equation}

\paragraph{The optimality gap.}
Note that \eqref{eq:radial-compression1} implies $\rho_R(\vzero)=\vzero$. 
Combining with \eqref{eq:hard-chain-origin}, we obtain
\begin{align*}
H_{\mU,\mu}(\vzero)=\bar f_T(\vzero)=-1.     
\end{align*}
Moreover, the lower bound of $\bar f_T(\vy)\ge-2T$ in \eqref{eq:hard-chain-infimum} implies
\[
    H_{\mU,\mu}(\vx)
    \ge\bar f_T(\mU^\top\rho_R(\vx))
    \ge-2T
    \qquad\text{for all }\vx\in\BR^d.
\]
Combining the above two results, we have
\begin{equation*}
    H_{\mU,\mu}(\vzero)
    -\inf_{\vx\in\BR^d}H_{\mU,\mu}(\vx)
    \le -1+2T   \le2T.
\end{equation*}

\paragraph{The lower bound of the Gradient.}
According to the chain rule and \eqref{eq:radial-jacobian}, the gradient of~$H_{\mU,\mu}$ has the form of
\begin{equation}
    \nabla H_{\mU,\mu}(\vx)
    =
    D\rho_R(\vx)\mU\nabla\bar f_T(\vy)
    +\mu\vx.
    \label{eq:global-gradient-formula}
\end{equation}
Since we have assumed there exists some $i\in[T]$ such that
$|y_i|=\left| [\mU^\top\rho_R(\vx)]_i\right|\le1$, 
we denote
\begin{equation*}
    i_\star:=\min\{i\in[T]:|y_i|\le1\}.
\end{equation*}
If $i_\star=1$, then \eqref{eq:hard-chain-gradient-coordinate} and the non-negativity of $\psi'$ and $\phi$ from Lemma~\ref{lem:gate-properties}(b) and (c) imply
\begin{align}\label{eq:d-f-1-phi}    
    \frac{\partial\bar f_T(\vy)}{\partial y_1}
    \le-\phi'(y_1).
\end{align}
If $i_\star\ge2$, the minimality of $i_\star$ gives $|y_{i_\star-1}|>1$. 
Hence, Lemma~\ref{lem:gate-properties}(a) implies
\begin{align*}
\psi(-y_{i_\star-1})+\psi(y_{i_\star-1})=1.
\end{align*}
Substituting the above result into \eqref{eq:hard-chain-gradient-coordinate} also yields \eqref{eq:d-f-1-phi}. 
Thus, Lemma~\ref{lem:gate-properties}(d) implies
\begin{equation}
    \frac{\partial\bar f_T(\vy)}{\partial y_{i_\star}}
    \le-\phi'(y_{i_\star})
    <-\frac25.
    \label{eq:global-unfinished-gradient}
\end{equation}
In the case of $\Norm{\vx}\ge20\sqrt T/\mu$, we apply 
equations (\ref{eq:global-gradient-formula}), (\ref{eq:radial-jacobian-bound}), and (\ref{eq:global-chain-bounds}) to achieve
\begin{align}\label{eq:global-large-radius-lower}
\begin{split}  
    \Norm{\nabla H_{\mU,\mu}(\vx)}
    &\ge 20\sqrt{T}
    -\Norm{D\rho_R(\vx)\mU\nabla\bar f_T(\vy)} \\
    &\ge 20\sqrt{T}
    -\Norm{\nabla\bar f_T(\vy)} \\
    &\ge  20\sqrt{T} -9\sqrt T    \\
    &\ge  11.   
\end{split}
\end{align}
In the case of $\Norm{\vx}<20\sqrt T/\mu$. The setting $R=1000\sqrt T/\mu$ implies $\Norm{\vx}/R<1/50$. Therefore, \eqref{eq:radial-alpha} implies
\begin{equation}
    \alpha(\vx)>
    \left(1+\frac1{2500}\right)^{-1/2}
    >\frac{99}{100}
    \qquad \text{and} \qquad
    \frac1{\alpha(\vx)}<\frac{101}{100}.
    \label{eq:global-alpha-near}
\end{equation}
Let $\vu_{i_\star}\in\BR^d$ be the $i_\star$th column of $\mU$.
Recall that $\mU^\top\mU=\mI_T$, we then apply equations (\ref{eq:global-gradient-formula}) and~(\ref{eq:radial-jacobian}) to~give
\begin{align}
    \langle\vu_{i_\star},\nabla H_{\mU,\mu}(\vx)\rangle
    = 
    \frac{\alpha(\vx)\partial\bar f_T(\vy)}{\partial y_{i_\star}}
    - \frac{(\alpha(\vx))^3}{R^2}
    \langle\vu_{i_\star},\vx\rangle
    \langle\vx,\mU\nabla\bar f_T(\vy)\rangle
    + \mu\langle\vu_{i_\star},\vx\rangle.
    \label{eq:global-directional-gradient}
\end{align}
Note that the definition of $i_\star$ indicates  $y_{i_\star}=\alpha(\vx)\langle\vu_{i_\star},\vx\rangle$ and $|y_{i_\star}|\le1$, which results
\begin{equation}
    |\langle\vu_{i_\star},\vx\rangle|
    \le\frac1{\alpha(\vx)}.
    \label{eq:global-coordinate-bound}
\end{equation}
For the second term on the right-hand side of inequality (\ref{eq:global-alpha-near}), we have
\begin{align}\label{eq:global-compression-error}
\begin{split}    
& \frac{(\alpha(\vx))^3}{R^2}
\left| \langle\vu_{i_\star},\vx\rangle
\langle\vx,\mU\nabla\bar f_T(\vy)\rangle \right|  
\le \frac{(\alpha(\vx))^3}{R^2}
\left| \langle\vu_{i_\star},\vx\rangle \right|
\left\| \vx\right\| \left\|\mU\nabla\bar f_T(\vy) \right\|  \\
\le & \frac{(\alpha(\vx))^2\left\| \vx\right\| }{R^2}
\left\|\mU\nabla\bar f_T(\vy) \right\| 
%\\ %\leq & \frac{(\alpha(\vx))^2}{R^2} \Norm{\vx}\Norm{\nabla\bar f_T(\vy)}
\le \frac{9\sqrt T\Norm{\vx}}{R^2}
\leq \frac{9\sqrt T(20\sqrt T/\mu)}{(1000\sqrt T/\mu)^2}
=\frac{\mu}{5000}\le \frac{2}{10^{6}},
\end{split}
\end{align}
where the first inequality is based on Cauchy--Schwarz inequality,
the second inequality is based on \eqref{eq:global-coordinate-bound},
the third inequality is based on equation (\ref{eq:global-chain-bounds}) and the definition (\ref{eq:radial-alpha}),
and the last inequality is based on the condition $\Norm{\vx}<20\sqrt T/\mu$ and the setting $R=1000\sqrt T/\mu$.
For the last term on the right-hand side of inequality (\ref{eq:global-alpha-near}), \eqref{eq:global-alpha-near} and the setting $\mu\le10^{-2}$ implies
\begin{align}\label{eq:global-regularizer-error}
    \mu|\langle\vu_{i_\star},\vx\rangle|
    &\le\frac{\mu}{\alpha(\vx)}
    <\frac{101}{100}\mu \leq \frac{101}{10000}.
\end{align}
Combining equations (\ref{eq:global-unfinished-gradient}), (\ref{eq:global-compression-error}), (\ref{eq:global-regularizer-error}), and (\ref{eq:global-alpha-near}) with \eqref{eq:global-directional-gradient}, we obtain
\begin{align*}
    \left|
    \langle\vu_{i_\star},\nabla H_{\mU,\mu}(\vx)\rangle
    \right|
    \ge \frac{2\alpha(\vx)}5 - 2\times10^{-6} -\frac{101}{10000}
    > \frac25\cdot\frac{99}{100}    
    -2\times10^{-6} -\frac{101}{10000} 
    > \frac14.
    %\label{eq:global-directional-lower}
\end{align*}
Since we have $\Norm{\vu_{i_\star}}=1$, the Cauchy--Schwarz inequality yields
\begin{equation}
    \Norm{\nabla H_{\mU,\mu}(\vx)}
    \ge
    \left|
    \langle\vu_{i_\star},\nabla H_{\mU,\mu}(\vx)\rangle
    \right|
    >\frac14.
    \label{eq:global-gradient-small-radius}
\end{equation}

Combining \eqref{eq:global-large-radius-lower} and \eqref{eq:global-gradient-small-radius}, we conclude that whenever there exists some $i\in[T]$ such that \mbox{$|[\mU^\top\rho_R(\vx)]_i|\le1$,} it holds $\Norm{\nabla H_{\mU,\mu}(\vx)}\ge 1/4$.

\paragraph{The upper bound for the Hessian.}
We now establish the gradient-dependent upper bound for the Hessian. 
For all vectors $\vh,\vk\in\BR^d$, the chain rule implies
\begin{align}
\begin{split}  
    D^2H_{\mU,\mu}(\vx)[\vh,\vk]
 =  \left\langle
    \nabla^2\bar f_T(\vy)
    \mU^\top D\rho_R(\vx)[\vh],
    \mU^\top D\rho_R(\vx)[\vk]
    \right\rangle + \left\langle
    \mU\nabla\bar f_T(\vy),
    D^2\rho_R(\vx)[\vh,\vk]
    \right\rangle + \mu\langle\vh,\vk\rangle.
\end{split}
\end{align}
Taking the supremum on $\vh$ and $\vk$ over the unit ball and  using the fact $\Norm{\mU}=1$, we obtain
\begin{align}\label{eq:global-uniform-hessian}
\begin{split}    
     \Norm{\nabla^2H_{\mU,\mu}(\vx)}
\le &{} \Norm{\nabla^2\bar f_T(\vy)}\Norm{D\rho_R(\vx)}^2+\Norm{\nabla\bar f_T(\vy)}\Norm{D^2\rho_R(\vx)}+\mu \\
\le &{}  58+\frac{27\sqrt T}{R}+\mu 
\leq 58+\frac{1027\mu}{1000}  < 59,
\end{split}
\end{align}
where the second inequality is based on equations (\ref{eq:global-chain-bounds}), (\ref{eq:radial-jacobian-bound}), and (\ref{eq:radial-second-bound}), the last two inequalities are based on the settings $R=1000\sqrt T/\mu$ and $\mu\le10^{-2}$.

If there exists some $i\in[T]$ such that $|y_i|\le1$.
According to \eqref{eq:global-hard-gradient} from Lemma \ref{lem:global-hard-properties}, we have $\Norm{\nabla H_{\mU,\mu}(\vx)}\ge1/4$. Combining this with \eqref{eq:global-uniform-hessian} and the setting $\mu<1$ yields
\begin{equation}
    \Norm{\nabla^2H_{\mU,\mu}(\vx)}
    <59
    <240\cdot\frac14
    \le240\left(
    \Norm{\nabla H_{\mU,\mu}(\vx)}
    +\mu\log\frac1\mu
    \right).
    \label{eq:global-curvature-case-one}
\end{equation}

We then focus on the case that $|y_i|>1$ holds for all $i\in[T]$.
If $\Norm{\vx}\ge20\sqrt T/\mu$, then equations~(\ref{eq:global-uniform-hessian}) and (\ref{eq:global-large-radius-lower}) imply
\begin{align}
    \Norm{\nabla^2H_{\mU,\mu}(\vx)}
    &<59\le 240\cdot 11
    \le 240 \Norm{\nabla H_{\mU,\mu}(\vx)}\leq 240\left(\Norm{\nabla H_{\mU,\mu}(\vx)}+\mu\log \frac1\mu
    \right). \label{eq:global-curvature-case-two}        
\end{align}
If $\Norm{\vx}<20\sqrt T/\mu$, then \eqref{eq:global-alpha-near} gives
\begin{align*}%\label{eq:alpha-ul}
    \alpha(\vx)>\frac{99}{100} \qquad\text{and}\qquad 
    \frac{1}{\alpha(\vx)}<\frac{100}{99}.
\end{align*}
The condition $|y_i|>1$ holds for all $i\in[T]$ indicates that we can combine \eqref{eq:hard-chain-saturated-gradient} and  Lemma~\ref{lem:gate-properties}(c) to achieve
\begin{equation}
    \Norm{\nabla^2\bar f_T(\vy)}
    =
    \max_{i\in[T]}|\phi''(y_i)|
    \le2\max_{i\in[T]}\phi'(y_i).
    \label{eq:global-completed-hessian-bound}
\end{equation}
We denote the index $j\in[T]$ such that $\phi'(y_j)=\max_{i\in[T]}\phi'(y_i)$ and analyze the right-hand side of inequality (\ref{eq:global-completed-hessian-bound}) by the following two cases.
\begin{enumerate}
\item If $|y_j|\ge\log(1/\mu)$, then Lemma~\ref{lem:gate-properties}(c) implies
\begin{equation}
    \max_{i\in[T]}\phi'(y_i)
    =
    \operatorname{sech}^2(y_j)
    \le4\mathrm e^{-2|y_j|}
    \le4\mu^2
    \le\mu\log\frac1\mu,
    \label{eq:global-phi-large-coordinate}
\end{equation}
where the first inequality follows from $\cosh t\ge\mathrm e^{|t|}/2$, and the last inequality uses $\mu\le10^{-2}$.
\item If $|y_j|<\log(1/\mu)$, then the identity $y_j=\alpha(\vx)\langle\vu_j,\vx\rangle$ implies
\begin{equation}
    |\langle\vu_j,\vx\rangle|
    <
    \frac{\log(1/\mu)}{\alpha(\vx)}.
    \label{eq:global-log-coordinate}
\end{equation}
Following the derivation of \eqref{eq:global-directional-gradient} by replacing $\vu_{i_\star}$ with $\vu_{j}$, we obtain
\begin{align*}
    \langle \vu_j ,\nabla H_{\mU,\mu}(\vx)\rangle  
    &=\frac{\alpha(\vx)\partial\bar f_T(\vy)}{\partial y_{j}}-\frac{(\alpha(\vx))^3}{R^2}\langle\vu_j,\vx\rangle\langle\vx,\mU\nabla\bar f_T(\vy)\rangle
    +\mu\langle\vu_j,\vx\rangle \\
    &=-\alpha(\vx)\phi'(y_j)-\frac{(\alpha(\vx))^3}{R^2}\langle\vu_j,\vx\rangle\langle\vx,\mU\nabla\bar f_T(\vy)\rangle
    +\mu\langle\vu_j,\vx\rangle,
\end{align*}
where the last step is based on \eqref{eq:hard-chain-saturated-gradient}.
Since we have $\phi'(y_j)=\max_{i\in[T]}\phi'(y_i)$ and~$\Norm{\vu_j}=1$, applying the triangle inequality on  the above equation implies
\begin{align}\label{eq:alpha-max-phi-yi}
    \alpha(\vx)\max_{i\in[T]}\phi'(y_i)\le\Norm{\nabla H_{\mU,\mu}(\vx)}+\frac{(\alpha(\vx))^3}{R^2}|\langle\vu_j,\vx\rangle|
    \left|\langle\vx,\mU\nabla\bar f_T(\vy)\rangle\right| + \mu|\langle\vu_j,\vx\rangle|.
\end{align}
For the second term on the right-hand side of \eqref{eq:alpha-max-phi-yi},
applying Cauchy--Schwarz inequality, \eqref{eq:global-chain-bounds} with the fact $\mU^\top\mU=\mI_T$ imply
\begin{align*}
|\langle\vx,\mU\nabla\bar f_T(\vy)\rangle|
\le \Norm{\vx}\Norm{\nabla\bar f_T(\vy)}\le9\sqrt T\Norm{\vx}.    
\end{align*}
Combining with \eqref{eq:global-log-coordinate}, we obtain
\begin{align}\label{eq:alpha-max-phi-yi-1}
\begin{split}    
   & \frac{(\alpha(\vx))^3}{R^2}|\langle\vu_j,\vx\rangle|\left|\langle\vx,\mU\nabla\bar f_T(\vy)\rangle\right| \\
\le & \frac{9(\alpha(\vx))^2\sqrt T\Norm{\vx}\log(1/\mu)}{R^2} \\
< & \frac{9\mu\log(1/\mu)}{50000}<\frac{\mu\log(1/\mu)}{5000},
\end{split}
\end{align}
where the second inequality follows from $\alpha(\vx)\le1$, $\Norm{\vx}<20\sqrt T/\mu$, and $R=1000\sqrt T/\mu$.
Substituting equations (\ref{eq:alpha-max-phi-yi-1}) and (\ref{eq:global-log-coordinate}) into \eqref{eq:alpha-max-phi-yi}, we obtain
\begin{align}
    \alpha(\vx)\max_{i\in[T]}\phi'(y_i)
    \le
    \Norm{\nabla H_{\mU,\mu}(\vx)}
    +
    \left(\frac1{5000}+\frac1{\alpha(\vx)}
    \right)\mu\log\frac1\mu.
    \label{eq:global-phi-start-bound}
\end{align}
Dividing \eqref{eq:global-phi-start-bound} by $\alpha(\vx)>0$ and applying \eqref{eq:global-alpha-near}, we obtain
\begin{align}\label{eq:global-phi-small-coordinate}
\begin{split}    
    \max_{i\in[T]}\phi'(y_i)
    &\le \frac1{\alpha(\vx)}
    \Norm{\nabla H_{\mU,\mu}(\vx)}+
    \left(\frac1{5000\alpha(\vx)} + \frac1{(\alpha(\vx))^2} \right)\mu\log\frac1\mu \\
    &\le 2\left(
    \Norm{\nabla H_{\mU,\mu}(\vx)}
    +\mu\log\frac1\mu
    \right).
\end{split}
\end{align}
\end{enumerate}
Combining equations (\ref{eq:global-phi-large-coordinate}) and (\ref{eq:global-phi-small-coordinate}), we conclude that
\begin{equation}
    \max_{i\in[T]}\phi'(y_i)
    \le
    2\left(
    \Norm{\nabla H_{\mU,\mu}(\vx)}
    +\mu\log\frac1\mu
    \right).
    \label{eq:global-max-phi-bound}
\end{equation}
We now return to the bound of Hessian. The first inequality in \eqref{eq:global-uniform-hessian} implies
\begin{align}
\begin{split}
    \Norm{\nabla^2H_{\mU,\mu}(\vx)}
    &\le
    \Norm{\nabla^2\bar f_T(\vy)}\Norm{D\rho_R(\vx)}^2+\Norm{\nabla\bar f_T(\vy)}\Norm{D^2\rho_R(\vx)}+\mu\\
    &\le2\max_{i\in[T]}\phi'(y_i)+\frac{27\sqrt T}{R}+\mu
    \\
    &\le4\left(\Norm{\nabla H_{\mU,\mu}(\vx)}+\mu\log\frac1\mu\right)+\frac{1027}{1000}\mu
    \\
    &\le\frac{5027}{1000}\left(\Norm{\nabla H_{\mU,\mu}(\vx)}+\mu\log\frac1\mu\right)
    \\
    &<240\left(\Norm{\nabla H_{\mU,\mu}(\vx)}+\mu\log\frac1\mu\right),
\end{split}
\label{eq:global-curvature-case-three}
\end{align}
where the second inequality is based on 
equations (\ref{eq:global-completed-hessian-bound}),  
(\ref{eq:radial-jacobian-bound}), (\ref{eq:global-chain-bounds}),and  (\ref{eq:radial-second-bound});
the third inequality  follows from \eqref{eq:global-max-phi-bound} and setting $R=1000\sqrt T/\mu$;
the inequality uses the fact $\mu\le10^{-2}$ which implies $\mu\le\mu\log(1/\mu)$.

Combining equations (\ref{eq:global-curvature-case-one}), (\ref{eq:global-curvature-case-two}), and (\ref{eq:global-curvature-case-three}), we conclude that for every $\vx\in\BR^d$, it holds
\begin{equation*}
    \Norm{\nabla^2H_{\mU,\mu}(\vx)}
    \le
    240\left(
    \Norm{\nabla H_{\mU,\mu}(\vx)}
    +\mu\log\frac1\mu
    \right).
    %\label{eq:global-curvature-proof}
\end{equation*}
This completes the proof.
\end{proof}

\subsection{The Proof of Lemma~\ref{lem:hard-scaling}}
\begin{proof}
Recall that the definition of $f_{\mU}(\vx)=aH_{\mU,\mu}(\vx/b)$ in \eqref{eq:scaled-hard-function} implies 
\begin{align}
    \nabla f_{\mU}(\vx)
    =\frac{a}{b}\nabla H_{\mU,\mu}\left(\frac{\vx}{b}\right)
\qquad\text{and}\qquad
    \nabla^2 f_{\mU}(\vx) = \frac{a}{b^2}\nabla^2 H_{\mU,\mu}\left(\frac{\vx}{b}\right).
    \label{eq:scaled-hessian-formula}
\end{align}
We verify the properties of the function $f_{\mU}(\vx)$ as follows.

\paragraph{The proof of statement (a).}

According to \eqref{eq:global-uniform-hessian} and \eqref{eq:scaled-hessian-formula}, it holds
\begin{align}
    \Norm{\nabla^2 f_{\mU}(\vx)}
    &\leq \frac{59a}{b^2}
    =\frac{L_0+L_1\epsilon}{2}.
    \label{eq:scaled-uniform-hessian-proof}
\end{align}

Moreover, combining the gradient-dependent curvature bound (\ref{eq:global-hard-curvature}) and (\ref{eq:scaled-hessian-formula}), we have
\begin{align}\label{eq:scaled-hessian-intermediate}
\begin{split}    
& \Norm{\nabla^2 f_{\mU}(\vx)} 
= \frac{a}{b^2}\Norm{\nabla^2 H_{\mU,\mu}\left(\frac{\vx}{b}\right)} \\
\leq &  \frac{240a}{b^2}\left(
    \Norm{\nabla H_{\mU,\mu}\left(\frac{\vx}{b}\right)}
    +\mu\log\frac1\mu
    \right) \\
=&  \frac{240}{b}\Norm{\nabla f_{\mU}(\vx)}
    +\frac{240a\mu}{b^2}\log\frac1\mu \\
=& \frac{15(L_0+L_1\epsilon)}{118\epsilon}\Norm{\nabla f_{\mU}(\vx)}+\frac{120(L_0+L_1\epsilon)\mu}{59}\log\frac1\mu,
\end{split}
\end{align}
where the last step follows from the setting $b=1888\epsilon/(L_0+L_1\epsilon)$ and $a/b=16\epsilon$ in \eqref{eq:hard-scaling-parameters}.
For the term $\mu\log(1/\mu)$, recall we set
\begin{align}
    \delta_0:=\min\left\{10^{-2},\frac{L_0}{16(L_0+L_1\epsilon)}\right\}
    \qquad \text{and} \qquad
    \mu=\frac{\delta_0}{4\log(\mathrm e/\delta_0)}
    \label{eq:hard-scaling-parameters1}
\end{align}
in \eqref{eq:hard-scaling-parameters}, which implies
\begin{align*}
    \log\frac1\mu
    =
    \log\frac1{\delta_0}
    +\log\left(4\log\frac{\mathrm e}{\delta_0}\right)
    \le 2\log\frac1{\delta_0}+\log4 <2\log\frac{\mathrm e}{\delta_0},
    %\label{eq:scaled-mu-log}
\end{align*}
where the first inequality uses $\log t\le t-1$ for all $t>0$. 
Combining with the definition of $\delta_0$ and $\mu$ in \eqref{eq:hard-scaling-parameters1}, it holds
\begin{equation}
    \mu\log\frac1\mu
    \le\frac{\delta_0}{2}
    \le\frac{L_0}{32 (L_{0}+L_1\epsilon)}.
    \label{eq:scaled-mu-bound}
\end{equation} 

Substituting \eqref{eq:scaled-mu-bound} into \eqref{eq:scaled-hessian-intermediate}, we obtain 
\begin{align}
\begin{split}
    \Norm{\nabla^2 f_{\mU}(\vx)}
    &\leq \frac{15(L_0+L_1\epsilon)}{118\epsilon}
    \Norm{\nabla f_{\mU}(\vx)}
    +\frac{15L_0}{236}.
\end{split}
\label{eq:scaled-hessian-intermediate2}
\end{align}

If $\Norm{\nabla f_{\mU}(\vx)}\geq\epsilon$, then \eqref{eq:scaled-uniform-hessian-proof} implies
\begin{align*}
    \Norm{\nabla^2 f_{\mU}(\vx)}
    \leq\frac12(L_0+L_1\epsilon)\leq\frac12\left(L_0+L_1\Norm{\nabla f_{\mU}(\vx)}\right).
\end{align*}
Otherwise, \eqref{eq:scaled-hessian-intermediate2} implies
\begin{align*}
    \Norm{\nabla^2 f_{\mU}(\vx)}
    \leq\frac{45L_0}{236}
    +\frac{15L_1}{118}
    \Norm{\nabla f_{\mU}(\vx)}\leq\frac12\left(L_0+L_1\Norm{\nabla f_{\mU}(\vx)}\right).
\end{align*}
Therefore, for all $\vx\in\BR^d$, it holds
\begin{equation}
    \Norm{\nabla^2 f_{\mU}(\vx)}
    \leq\frac12\left(
    L_0+L_1\Norm{\nabla f_{\mU}(\vx)}
    \right).
\end{equation}

Following Corollary A.4 of \citet{zhang2020improved} by replacing their parameters $L_0$, $L_1$, and $c$ with~$L_0/2$, $L_1/2$, and $1$, respectively, we obtain
\begin{align*}   
    \Norm{\nabla f_{\mU}(\vy)-\nabla f_{\mU}(\vx)}
    &\le
    \left(
    L_0+\frac{(\mathrm e-1)L_1
    \Norm{\nabla f_{\mU}(\vx)}}{2}
    \right)\Norm{\vy-\vx} \\
    &\le
    \left(
    L_0+L_1\Norm{\nabla f_{\mU}(\vx)}
    \right)\Norm{\vy-\vx}
\end{align*}
for all $\vx,\vy\in\BR^d$ with $\Norm{\vy-\vx}\le 1/L_1\le 2/L_1$.
Thus, the function $f_{\mU}$ is $(L_0,L_1)$-smooth.

\paragraph{The proof of statement (b).}
Since $a,b>0$, the definition of the function $f_{\mU}$ in \eqref{eq:scaled-hard-function} and the initial optimality gap bound \eqref{eq:global-hard-gap} imply
\begin{align*}
    f_{\mU}(\vzero)
    -\inf_{\vx\in\BR^d}f_{\mU}(\vx)
    &=
    a\left(
    H_{\mU,\mu}(\vzero)
    -\inf_{\vx\in\BR^d}H_{\mU,\mu}\left(\frac{\vx}{b}\right)
    \right) \\
    &=
    a\left(
    H_{\mU,\mu}(\vzero)
    -\inf_{\vx\in\BR^d}H_{\mU,\mu}(\vx)
    \right)
    \le2aT,
    %\label{eq:scaled-initial-gap}
\end{align*}
where the second equality follows from the fact that the mapping $\vx\mapsto\vx/b$ is a bijection on $\BR^d$.

\paragraph{The proof of statement (c).}
Let $i$ be the index $i\in[T]$ such that
\[
    \left|
    \left[
    \mU^\top\rho_R\left(\frac{\vx}{b}\right)
    \right]_i
    \right|\le1.
\]
By the gradient formula in \eqref{eq:scaled-hessian-formula} and \eqref{eq:global-hard-gradient}, we obtain
\begin{equation*}
    \Norm{\nabla f_{\mU}(\vx)}
    =
    \frac{a}{b}
    \Norm{\nabla H_{\mU,\mu}\left(\frac{\vx}{b}\right)}
    \ge
    \frac{a}{4b}
    =
    4\epsilon,
\end{equation*}
where the last equality follows from the setting $a/b=16\epsilon$ in \eqref{eq:hard-scaling-parameters}. 
\end{proof}

\subsection{Discussion of Related Work}\label{appendix:disscusion}
Recall that the construction of \citet{arjevani2023lower} depends on the component functions
\[
\Psi(t)=
\begin{cases}
0,&t\le1/2,\\
\exp\!\left(1-\dfrac1{(2t-1)^2}\right),&t>1/2,
\end{cases}
\qquad
\text{and}
\qquad
\Phi(t)=\sqrt\mathrm{e}\int_{-\infty}^t \exp\left(-\frac{\tau^2}{2}\right)\,\mathrm{d}\tau,
\]
which holds 
\[
\Phi'(t)=\sqrt \mathrm{e}\cdot \exp\left(-\frac{t^2}{2}\right)
\qquad\text{and}\qquad
\Phi''(t)=-t\Phi'(t).
\]
We consider their hard instance in  the one-dimensional case, i.e., 
\begin{align*}
    f_1(t)=-\Phi(t),
\end{align*}
which holds
\begin{align}\label{eq:f1}
\lim_{t\to+\infty} \frac{|f''_1(t)|}{|f'_1(t)|} 
= \lim_{t\to+\infty} \frac{|\Phi''(t)|}{|\Phi'(t)|} 
= \lim_{t\to+\infty} |t| =  +\infty.
\end{align}
On the other hand, if a function $f:\BR^d\to\BR$ is $(0, L_1)$-smooth, then it holds
\begin{align}\label{eq:diff-f-th}
\left\| \frac{\nabla f(\vx+t\vh)-\nabla f(\vx)}{t} \right\|
\le L_1\|\nabla f(\vx)\|    
\end{align}
for all $\vh\in\BR^d$ with $\Norm{\vh}=1$ and $t\in\BR$ with $|t|\in(0, 1/L_1]$.
Taking $t\to 0$, \eqref{eq:diff-f-th} implies
\begin{align*}
\left\| \nabla^2 f(\vx) \right\|
= \sup_{\vh\in\BR^d,\,\Norm{\vh}=1}\left\| \nabla^2 f(\vx)\vh \right\|
\le L_1\|\nabla f(\vx)\|.
\end{align*}
Combining with \eqref{eq:f1}, we know that the function $f_1(\cdot)$ does not hold $(0, L_1)$-smoothness like the instance~$\bar f_T(\cdot)$ used in our construction.

%%%%%%%%%%%%%%%%%%%%%%%%%%%%%%%%%%%%%%%%%%%

%section 4

%%%%%%%%%%%%%%%%%%%%%%%%%%%%%%%%%%%%%%%%%%%%

\section{The Proofs for Results in Section \ref{sec:lower-bound}}

This section provides the proofs of our lower bound.

\subsection{The Proof of Lemma~\ref{lem:information-hiding}}
We first introduce the notations
\begin{equation}
    \chi(t):=1-S(4|t|-1)
    \qquad \text{and}\qquad
    \Theta_i(\vy):=\prod_{k=i}^{T}\chi(y_k),
    \label{eq:masking-functions}
\end{equation}
where 
\begin{equation*}
    S(u)= \begin{cases}
        0, & u\le 0,\\
        10u^3-15u^4+6u^5, & u\in(0,1),\\
        1, & u\ge 1,
    \end{cases}
\end{equation*}
follows definition in \eqref{eq:gate-S}.
Hence, \eqref{eq:masking-Q} implies we can write
\begin{equation}
    Q_T(\vy)=\sum_{i=1}^{T}A_i(\vy)\Theta_i(\vy).
    \label{eq:masking-Q-decomposition}
\end{equation}

We present several lemmas for the proof of Lemma~\ref{lem:information-hiding}.
\begin{lemma}\label{eq:chi-flat-derivative}
Based on the definition of $\chi(t)$ in \eqref{eq:masking-functions}, for all $t\in\BR$, it holds 
\begin{align} \label{eq:chi-derivative-bound}
0\le\chi(t)\le1
\qquad\text{and}\qquad
|\chi'(t)|^2\le14400\bigl(1-(\chi(t))^2\bigr).    
\end{align}
\end{lemma}
\begin{proof}    
According to the definition of $\chi(t)$ in \eqref{eq:masking-functions}, it holds
\begin{equation}\label{eq:chi-flat-values}
    \chi(t)= \begin{cases}
    1, & \text{if }|t|\le\dfrac14, \\[0.3cm]
    0, & \text{if }|t|\ge\dfrac12,
    \end{cases}
\end{equation}
which implies 
\begin{equation}\label{eq:chi-flat-values2}
    \chi'(t)=0
    \quad\text{whenever}~~
    |t|\le\frac14
    ~~\text{ or }~~|t|\ge\frac12.
\end{equation}
Moreover, the bounds on $S$ established in the proof of Lemma~\ref{lem:gate-properties}(b) imply
\begin{equation}
    0\le\chi(t)\le1
    \qquad\text{for all }t\in\BR.
    \label{eq:chi-range}
\end{equation}
Combining equations (\ref{eq:chi-flat-values}) and (\ref{eq:chi-range}) indicates \eqref{eq:chi-derivative-bound} holds whenever $|t|\le1/4$ or~$|t|\ge1/2$. 

It remains to prove \eqref{eq:chi-derivative-bound} holds when $|t|\in(1/4,1/2)$.
We let 
\begin{align*}
u:=4|t|-1\in(0,1)   
\end{align*}
for $t$ such that $|t|\in(1/4,1/2)$.
Applying the chain rule and \eqref{eq:S-prime} give
\begin{align*}
    |\chi'(t)| = 4S'(u) = 120u^2(1-u)^2.    
\end{align*}
Moreover, the definition of $S(u)$ in \eqref{eq:gate-S} gives
\[
    S(u)=u^3(10-15u+6u^2) = u^3(3(1-u)(3-2u)+1) \geq u^3.
\]
Hence, combining the facts $u\in(0,1)$ and $S(u)\in(0,1)$ with above two results leads to 
\begin{align*}
    |\chi'(t)|^2
    &=14400u^4(1-u)^4
    \nonumber\\
    &\le14400u^3
    \le14400S(u)(2-S(u))
    \nonumber\\
    &=14400\bigl(1-\chi(t)^2\bigr)
\end{align*}
whenever $|t|\in(1/4,1/2)$.
Thus, \eqref{eq:chi-derivative-bound} holds for every $t\in\BR$.
\end{proof}

\begin{lemma}
\label{lem:theta-gradient}
The functions $\Theta_i(\vy)$ defined in \eqref{eq:masking-functions} satisfy
\[
    \Norm{\nabla\Theta_i(\vy)}\le120
\]
for all $i\in[T]$ and $\vy\in\BR^T$.
\end{lemma}

\begin{proof}
The definition $\Theta_i(\vy):=\prod_{k=i}^{T}\chi(y_k)$ implies
\begin{equation}
    \frac{\partial\Theta_i(\vy)}{\partial y_j}
    =
    \begin{cases}
        0, & j<i,\\[0.1cm]
        \displaystyle
        \chi'(y_j)
        \prod_{\substack{k=i\\k\ne j}}^{T}
        \chi(y_k), & j\ge i.
    \end{cases}
    \label{eq:theta-coordinate-derivative}
\end{equation}
Consequently, Lemma \ref{eq:chi-flat-derivative} implies 
\begin{align*}
    \Norm{\nabla\Theta_i(\vy)}^2
    = \sum_{j=i}^{T}
    |\chi'(y_j)|^2
    \prod_{\substack{k=i\\k\ne j}}^{T}
    \chi(y_k)^2
    \le 14400\sum_{j=i}^{T}
    \bigl(1-\chi(y_j)^2\bigr)
    \prod_{\substack{k=i\\k\ne j}}^{T}
    \chi(y_k)^2 \le14400. 
\end{align*}
Taking square roots on  the above result yields
\begin{equation*}
    \Norm{\nabla\Theta_i(\vy)}\le120,
\end{equation*}
which completes the proof.
\end{proof}

\begin{lemma}
\label{lem:mask-one-step}
For a given $\vy=[y_1,\dots,y_T]\in\BR^T$,
suppose there exists some $s\in\{0,\ldots,T-1\}$ such that
\mbox{$|y_{s+1}|,\ldots,|y_T|\le1/4$}.
Then the stochastic component $\bar F_T$ defined in \eqref{eq:stochastic-chain} satisfies:
\begin{enumerate}[label=(\alph*),leftmargin=0.65cm]
    \item If $Z=0$, then
    $\operatorname{supp}(\nabla\bar F_T(\vy;0))
    \subseteq[s]$ and
    $\nabla\bar F_T(\vy;0)
    =\nabla\bar F_T(\Pi_s(\vy);0)$.

    \item If $Z=1$, then
    $\operatorname{supp}(\nabla\bar F_T(\vy;1))
    \subseteq[s+1]$ and
    $\nabla\bar F_T(\vy;1)
    =\nabla\bar F_T(\Pi_{s+1}(\vy);1)$.
\end{enumerate}
Here, we denote  $\Pi_s(\vy)=[y_1,\dots,y_s,0,\dots,0]^\top$.
\end{lemma}

\begin{proof}
The assumption on the index $s$ means
\begin{equation}
    |y_{s+1}|,\ldots,|y_T|\le\frac14.
    \label{eq:information-tail-assumption}
\end{equation}
According to equations (\ref{eq:masking-functions}) and (\ref{eq:chi-flat-values}), we have
\begin{equation*}
    \chi(y_k)=1
    \qquad \text{and}\qquad
    \chi'(y_k)=0
    \qquad \text{for all}~~k=s+1,\ldots,T.
    %\label{eq:tail-chi-flat}
\end{equation*}
Consequently, the definition of $\Theta_i(\vy)$ in \eqref{eq:masking-functions} implies
\begin{equation}
    \Theta_i(\vy)
    =\begin{cases}
    \displaystyle{\prod_{k=i}^{s}\chi(y_k)},  & \text{if}~i\le s+1, \\[0.35cm]
    1 & \text{otherwise}.
    \end{cases}
    \label{eq:theta-tail-reduction}
\end{equation}
Hence, we have
\begin{equation}
    \Theta_{s+1}(\vy)=1,
    \qquad\text{and}\qquad
    \nabla\Theta_{s+1}(\vy)=\vzero,
    \label{eq:theta-s-plus-one}
\end{equation}
and also holds
\begin{equation}
    \frac{\partial\Theta_i(\vy)}{\partial y_j}=0 \qquad
    \text{whenever} \quad i\le s+1 \quad\text{and}\quad j\ge s+1.
    \label{eq:theta-no-tail-derivative}
\end{equation}
Moreover, for all $i\ge s+2$, condition \eqref{eq:information-tail-assumption} gives $|y_{i-1}|\le1/4$ and $|y_i|\le1/4$. 
Hence, Lemma~\ref{lem:gate-properties}(a) implies
\[
    \psi(y_{i-1})=\psi(-y_{i-1})=\psi'(y_{i-1})=\psi'(-y_{i-1})=0.
\]
Hence, the definition of $A_i(\vy)$ in \eqref{eq:hard-zero-chain} implies
\begin{equation}
    A_i(\vy)=0
    \qquad\text{and}\quad
    \nabla A_i(\vy)=\vzero
    \qquad \text{for all}~~i\ge s+2.
    \label{eq:A-tail-zero}
\end{equation}

\paragraph{The proof of statement (a).}
For $Z=0$, equations~(\ref{eq:stochastic-chain}) and (\ref{eq:masking-Q-decomposition}) give
\begin{equation*}
    \bar F_T(\vy;0)
    =
    \sum_{i=1}^{T}
    A_i(\vy)\bigl(1-\Theta_i(\vy)\bigr).
    %\label{eq:F-zero-expansion}
\end{equation*}
Combining with equations (\ref{eq:theta-s-plus-one}) and (\ref{eq:A-tail-zero}), we have
\begin{equation}
    \nabla\bar F_T(\vy;0)
    =
    \sum_{i=1}^{s}
    \nabla\left(
    A_i(\vy)\bigl(1-\Theta_i(\vy)\bigr)
    \right).
    \label{eq:F-zero-first-s}
\end{equation}
Under condition (\ref{eq:information-tail-assumption}), 
we know that \eqref{eq:hard-zero-chain} indicates
all functions $A_1,\dots,A_s$ only depend on  $y_1,\ldots,y_s$,
and equations (\ref{eq:theta-tail-reduction}) and (\ref{eq:theta-no-tail-derivative}) imply that the function values and the gradients of all~$\Theta_1,\dots,\Theta_s$ only depend on the $y_1,\ldots,y_s$. 
Hence, \eqref{eq:F-zero-first-s} yields
\begin{align*}
    \operatorname{supp}
    \bigl(\nabla\bar F_T(\vy;0)\bigr)
    \subseteq[s]
    \qquad\text{and}\qquad
    \nabla\bar F_T(\vy;0)
    =\nabla\bar F_T(\Pi_s(\vy);0).
\end{align*}
%Here $[0]$ denotes the empty set. This establishes statement (a).

\paragraph{The proof of statement (b).}
For $Z=1$, equations~(\ref{eq:stochastic-chain}) and (\ref{eq:masking-Q-decomposition}) give
\begin{equation*}
    \bar F_T(\vy;1)
    =\sum_{i=1}^{T} A_i(\vy)\left(1+\left(\frac1q-1\right)\Theta_i(\vy)\right).
    %\label{eq:F-one-expansion}
\end{equation*}
Recall that \eqref{eq:A-tail-zero} indicates that all terms with $i\ge s+2$ in the above summation have zero gradient. Thus, it holds
\begin{equation}
    \nabla\bar F_T(\vy;1)
    =
    \sum_{i=1}^{s+1}
    \nabla\left(
    A_i(\vy)
    \left(
    1+\left(\frac1q-1\right)\Theta_i(\vy)
    \right)
    \right).
    \label{eq:F-one-first-s-plus-one}
\end{equation}
Under condition (\ref{eq:information-tail-assumption}), 
we know \eqref{eq:hard-zero-chain} implies all function $A_1,\dots,A_{s+1}$ only depends on~$y_1,\dots,y_{s+1}$,
and equations (\ref{eq:theta-tail-reduction}), (\ref{eq:theta-s-plus-one}), and (\ref{eq:theta-no-tail-derivative}) imply that the function values and the gradients of $\Theta_1,\dots,\Theta_{s+1}$ only depend on $y_1,\dots,y_{s+1}$. 
Hence, \eqref{eq:F-one-first-s-plus-one} yields
\begin{align*}
    \operatorname{supp}
    \bigl(\nabla\bar F_T(\vy;1)\bigr)
    \subseteq[s+1]
    \qquad\text{and}\qquad
    \nabla\bar F_T(\vy;1)
    =\nabla\bar F_T(\Pi_{s+1}(\vy);1).
\end{align*}
\end{proof}

We now present the proof of  Lemma~\ref{lem:information-hiding}.

\begin{proof}
We verify properties of functions $\bar F_T$ and $Q_T$.

\paragraph{The proof of statement (a).}
According to \eqref{eq:masking-Q-decomposition} and the product rule, we know that
\begin{equation}
    \nabla Q_T(\vy)
    =
    \sum_{i=1}^{T}
    \left(
    \Theta_i(\vy)\nabla A_i(\vy)
    +A_i(\vy)\nabla\Theta_i(\vy)
    \right).
    \label{eq:grad-Bi}
\end{equation}
We denote
\begin{equation*}
    m:=\operatorname{prog}_{1/2}(\vy)
    \in\{0,1,\ldots,T\}.
\end{equation*}
For any $m\in\{0,1,\ldots,T\}$ and $i\ge m+2$, we have $|y_{i-1}|\le1/2$. Hence, Lemma~\ref{lem:gate-properties}(a) and \eqref{eq:gate-pair-derivative-zero} implies
\[
    \psi(y_{i-1}) = \psi(-y_{i-1})
    =\psi'(y_{i-1})=\psi'(-y_{i-1})=0.
\]
Combining  the above results with \eqref{eq:hard-zero-chain} gives
$A_i(\vy)=0$ and $\nabla A_i(\vy)=\vzero$, which implies
\begin{equation}
    %\nabla\bigl(A_i(\vy)\Theta_i(\vy)\bigr)=\vzero
     \Theta_i(\vy)\nabla A_i(\vy)+A_i(\vy)\nabla\Theta_i(\vy)=\vzero
    ~~~~~\text{for all}~i\ge m+2.
    \label{eq:Bi-right-zero}
\end{equation}

In the case of $m=0$, combining equations (\ref{eq:grad-Bi}) and (\ref{eq:Bi-right-zero}) implies
\begin{align}\label{eq:Bi-left-zero-0}
    \nabla Q_T(\vy)
    = \Theta_1(\vy)\nabla A_1(\vy)+A_1(\vy)\nabla\Theta_1(\vy).
    %\nabla\bigl(A_1(\vy)\Theta_1(\vy)\bigr).    
\end{align}
In the case of $m\in[T]$, we have $|y_m|>1/2$. Hence, equations (\ref{eq:chi-flat-values}) and (\ref{eq:chi-flat-values2}) imply
\[
    \chi(y_m)=\chi'(y_m)=0.
\]
Note that for all $i\le m$, equations~(\ref{eq:masking-functions}) and (\ref{eq:theta-coordinate-derivative}) yield
\[
    \Theta_i(\vy)=\prod_{k=i}^{T}\chi(y_k)=0
    \qquad\text{and}\qquad
    \nabla\Theta_i(\vy)=\vzero,
\]
respectively. This implies
\begin{equation}
    \Theta_i(\vy)\nabla A_i(\vy)+A_i(\vy)\nabla\Theta_i(\vy)=\vzero
    ~~~~~\text{for all}~i\le m.
    \label{eq:Bi-left-zero}
\end{equation}
Combining equations (\ref{eq:Bi-right-zero}), (\ref{eq:Bi-left-zero-0}) and (\ref{eq:Bi-left-zero}), we obtain
\begin{equation}
    \nabla Q_T(\vy)=
    \begin{cases}
        \Theta_{m+1}(\vy)\nabla A_{m+1}(\vy)+A_{m+1}(\vy)\nabla\Theta_{m+1}(\vy)
        & \text{if}~m<T,\\[0.1cm]
        \vzero, & \text{if}~m=T.
    \end{cases}
    \label{eq:single-active-B}
\end{equation}

We now establish the upper bound on $\Norm{\nabla Q_T(\vy)}$ by considering the following three cases.
\begin{enumerate}[leftmargin=0.65cm]
\item In the case of $m=0$, we have $A_1(\vy)=-\phi(y_1)$, so that Lemma~\ref{lem:gate-properties}(c) implies
$|A_1(\vy)|<2$ and \mbox{$\Norm{\nabla A_1(\vy)}\le1$}.
By \eqref{eq:single-active-B}, \eqref{eq:chi-range}, and Lemma~\ref{lem:theta-gradient}, we obtain
\begin{align*}
    \Norm{\nabla Q_T(\vy)}
    &\le
    |\Theta_1(\vy)|\Norm{\nabla A_1(\vy)}+|A_1(\vy)|\Norm{\nabla\Theta_1(\vy)}
    \\
    &\le1+2\cdot120
    =241<249.
\end{align*}

\item In the case of $m\in[T-1]$, combining the definition of $A_i(\vy)$ in equation~(\ref{eq:hard-zero-chain}) with \eqref{eq:gate-pair-value} and Lemma~\ref{lem:gate-properties}(c) imply
\begin{equation}
    |A_{m+1}(\vy)|
    = |\psi(-y_{m})\phi(-y_{m+1})-\psi(y_{m})\phi(y_{m+1})|
    \le 2(\psi(-y_m)+\psi(y_m))
    \le 2.
    \label{eq:Aj-bound}
\end{equation}
Moreover, combining the definition of $A_i(\vy)$ in \eqref{eq:hard-zero-chain} with \eqref{eq:gate-pair-first}, \eqref{eq:gate-pair-value}, and Lemma~\ref{lem:gate-properties}(c), we have
\begin{align*}
    &\left|\frac{\partial A_{m+1}(\vy)}{\partial y_m}\right|
    = |\psi'(-y_{m})\phi(-y_{m+1})+\psi'(y_{m})\phi(y_{m+1})|
    \le
    2|\psi'(-y_m)+\psi'(y_m)|
    \le8,\\
    &\left|\frac{\partial A_{m+1}(\vy)}{\partial y_{m+1}}\right|
    = |\psi(-y_{m})\phi'(-y_{m+1})+\psi(y_{m})\phi'(y_{m+1})|
    \leq |\psi(-y_m)+\psi(y_m)|
    \le1.
\end{align*}
Since the function $A_{m+1}(\vy)$ depends only on $y_m$ and $y_{m+1}$, the above results lead to
\begin{equation}
    \Norm{\nabla A_{m+1}(\vy)}
    \le\sqrt{8^2+1^2}
    =\sqrt{65}<9.
    \label{eq:grad-Aj-bound}
\end{equation}
Therefore, combining \eqref{eq:single-active-B} with equations (\ref{eq:masking-functions}), (\ref{eq:chi-range}), (\ref{eq:grad-Aj-bound}), and (\ref{eq:Aj-bound}), and Lemma~\ref{lem:theta-gradient}, we have
\begin{align}
\begin{split}    
    \Norm{\nabla Q_T(\vy)}
    &\le
    |\Theta_{m+1}(\vy)|
    \Norm{\nabla A_{m+1}(\vy)}
    +|A_{m+1}(\vy)|
    \Norm{\nabla\Theta_{m+1}(\vy)} \\
    &\le9+2\cdot120
    =249.
    \label{eq:Q-gradient-bound}
\end{split}
\end{align}

\item 
In the case of $m=T$, \eqref{eq:single-active-B} directly gives 
$\Norm{\nabla Q_T(\vy)}=0\le249$.
\end{enumerate}
Therefore, the above derivations conclude $\Norm{\nabla Q_T(\vy)}\le249$ holds for all $\vy\in\BR^T$.

\paragraph{The proof of statement (b).}
Differentiating \eqref{eq:stochastic-chain} obtains
\begin{equation*}
    \vg_T(\vy;Z)
    =
    \nabla\bar f_T(\vy)
    +
    \left(\frac Zq-1\right)\nabla Q_T(\vy).
    %\label{eq:masked-stochastic-gradient}
\end{equation*}
The distribution $Z\sim\operatorname{Bernoulli}(q)$ indicates $\BE_Z[Z/q-1]=0$, resulting
\begin{equation*}
    \BE_Z[\vg_T(\vy;Z)]
    =
    \nabla\bar f_T(\vy).
    %\label{eq:masked-unbiasedness}
\end{equation*}

We then verify the property of probability-$q$ zero-chain for $\vg_T(\vy;Z)$. Fix given $\vy\in\BR^T$, we let 
\begin{align*}
s=\operatorname{prog}_{1/4}(\vy).     
\end{align*}
We show that $\vg_T(\vy;Z)$ is $q$-probability zero-chain by verifying conditions (\ref{eq:prob-zero-chain-no-progress}) and
(\ref{eq:prob-zero-chain-one-progress}) as follows.
\begin{enumerate}[leftmargin=0.65cm]
\item If $s<T$, the definition of $s$ implies $|y_{s+1}|,\ldots,|y_T|\le1/4$. Hence, when $Z=0$, we can apply Lemma~\ref{lem:mask-one-step}(a) to achieve
\begin{align}\label{eq:p-zero-chain-0}    
    \operatorname{prog}_0(\vg_T(\vy;0))
    \le s
    \qquad\text{and}\qquad
    \vg_T(\vy;0)
    =
    \vg_T(\Pi_s(\vy);0)
    \qquad \text{for all } \vy\in\BR^T.
\end{align}
Recall that $\mathbb P(Z=0)=1-q$, then \eqref{eq:p-zero-chain-0} verifies that condition \eqref{eq:prob-zero-chain-no-progress}  with~$\theta=q$.
When $Z=1$, we can apply Lemma~\ref{lem:mask-one-step}(b) to achieve
\begin{align}\label{eq:p-zero-chain-1}
    \operatorname{prog}_0(\vg_T(\vy;1))
    \le s+1 
    \qquad\text{and}\qquad
    \vg_T(\vy;1) =
    \vg_T(\Pi_{s+1}(\vy);1)
    \qquad \text{for all } \vy\in\BR^T.
\end{align}
According to Lemma~\ref{lem:mask-one-step}(a), we know that \eqref{eq:p-zero-chain-1} also holds when $Z=0$.
Hence, we have verified condition~(\ref{eq:prob-zero-chain-one-progress})  with $\theta=q$.

\item If $s=T$, we have $\operatorname{prog}_{1/4}(\vy)=T$ and $\Pi_{1+\operatorname{prog}_{1/4}}(\vy)=\Pi_{1+\operatorname{prog}_{1/4}}(\vy)=\vy$. Therefore, for all $\vy\in\BR^T$, all conditions
\begin{align*}    
   \operatorname{prog}_0(\vg_T(\vy;Z))
    \le \operatorname{prog}_{1/4}(\vy),\quad
    \vg_T(\vy;Z)
    = \vg_T\bigl(
    \Pi_{\operatorname{prog}_{1/4}(\vy)}(\vy);Z
    \bigr), \\
    \operatorname{prog}_0(\vg_T(\vy;Z))
    \le1+\operatorname{prog}_{1/4}(\vy),\quad
    \vg_T(\vy;Z)
    = \vg_T\bigl(
    \Pi_{1+\operatorname{prog}_{1/4}(\vy)}(\vy);Z
    \bigr)
\end{align*}
hold in both cases of $Z=0$ and $Z=1$.
This verifies equations (\ref{eq:prob-zero-chain-no-progress}) and (\ref{eq:prob-zero-chain-one-progress}).
\end{enumerate}
Hence, we conclude that $\vg_T$ is a probability-$q$ zero-chain. 
\end{proof}

\subsection{The Proof of Lemma~\ref{lem:oracle-validity}}
\begin{proof}
We let $\vy:=\mU^\top\rho_R(\vx/b)$. By the chain rule, the stochastic oracle defined in \eqref{eq:hard-stochastic-oracle} has the form of 
\begin{equation*}
    \mathsf O_{f_{\mU}}(\vx;Z)
    =
    \frac{a}{b}
    D\rho_R\left(\frac{\vx}{b}\right)^\top
    \mU\nabla\bar F_T(\vy;Z)
    +\frac{a\mu\vx}{b^2}.
\end{equation*}
Combining the above result  with equations (\ref{eq:global-hard-function}), (\ref{eq:stochastic-chain}), (\ref{eq:scaled-hessian-formula}), and (\ref{eq:global-gradient-formula}), we obtain
\begin{equation}
    \mathsf O_{f_{\mU}}(\vx;Z)
    = \nabla f_{\mU}(\vx) + 
    \frac{a}{b}\left(\frac Zq-1\right)
    D\rho_R\left(\frac{\vx}{b}\right)^\top
    \mU\nabla Q_T(\vy).
    \label{eq:oracle-noise-exact}
\end{equation}

\paragraph{Unbiasedness.}
Since $Z\sim\operatorname{Bernoulli}(q)$ with $q\in(0,1]$, taking the expectation on \eqref{eq:oracle-noise-exact} achieves
\begin{equation}
    \BE_Z[\mathsf O_{f_{\mU}}(\vx;Z)]
    =
    \nabla f_{\mU}(\vx).
    \label{eq:oracle-unbiased-proof}
\end{equation}

\paragraph{Bound on the $p$th central moment.}
Combining equations (\ref{eq:oracle-noise-exact}), (\ref{eq:radial-jacobian-bound}), and (\ref{eq:Q-gradient-bound}), we have
\begin{align}\label{eq:oracle-noise-pointwise}
\begin{split}    
    \Norm{\mathsf O_{f_{\mU}}(\vx;Z)-\nabla f_{\mU}(\vx)}
    &\le
    \frac{a}{b}
    \left|\frac Zq-1\right|
    \Norm{D\rho_R(\vx/b)}
    \Norm{\mU}
    \Norm{\nabla Q_T(\vy)} \\
    &\le
    16\epsilon\cdot249
    \left|\frac Zq-1\right|
    =
    3984\epsilon
    \left|\frac Zq-1\right|,
\end{split}
\end{align}
where we use $\Norm{\mU}=1$ and $a/b=16\epsilon$ from equation (\ref{eq:hard-scaling-parameters}). Since $Z\sim\operatorname{Bernoulli}(q)$ and~$p\in(1,2]$, we have
\begin{align}
    \BE_Z\left[\left|\frac Zq-1\right|^p\right]
    &=
    (1-q)^p q^{1-p}+1-q\le
    q^{1-p}+1
    \le2q^{1-p},
    \label{eq:oracle-bernoulli-upper}
\end{align}
where the inequalities follow from $0<q\le1$ and $p>1$. Consequently, combining equations (\ref{eq:oracle-noise-pointwise}) and (\ref{eq:oracle-bernoulli-upper}) imply
\begin{equation}
    \BE_Z\left[
    \Norm{\mathsf O_{f_{\mU}}(\vx;Z)-\nabla f_{\mU}(\vx)}^p
    \right]
    \le
    2(3984\epsilon)^p q^{1-p}.
    \label{eq:oracle-noise-before-q}
\end{equation}
If $q<1$, then $\sigma>0$ and the definition of $q$ gives $q^{p-1}=2(3984)^p\epsilon^p\sigma^{-p}$.
Substituting this identity into \eqref{eq:oracle-noise-before-q}, we obtain
\begin{align}
    \BE_Z\left[
    \Norm{\mathsf O_{f_{\mU}}(\vx;Z)-\nabla f_{\mU}(\vx)}^p
    \right]
    \le 2(3984\epsilon)^p q^{1-p}=2(3984\epsilon)^p\cdot\frac{\sigma^p}{2(3984)^p\epsilon^p}=\sigma^p.
    \label{eq:oracle-moment-q-less-one}
\end{align}
If $q=1$, then $Z=1$ almost surely. Hence, \eqref{eq:oracle-unbiased-proof} implies
\begin{equation}
    \BE_Z\left[
    \Norm{\mathsf O_{f_{\mU}}(\vx;Z)-\nabla f_{\mU}(\vx)}^p
    \right]
    =0\le\sigma^p.
    \label{eq:oracle-moment-q-one}
\end{equation}
Combining equations (\ref{eq:oracle-moment-q-less-one}) and (\ref{eq:oracle-moment-q-one}), we conclude that for every $\vx\in\BR^d$, it holds
\begin{equation}
    \BE_Z\left[
    \Norm{\mathsf O_{f_{\mU}}(\vx;Z)-\nabla f_{\mU}(\vx)}^p
    \right]
    \le\sigma^p.
\end{equation}
\end{proof}

\subsection{The Proof of Lemma~\ref{lem:randomized-hardness}}
\begin{proof}
According to Lemma~\ref{lem:information-hiding}, we know that 
the gradient evaluation $\vg_T(\vy;Z)$ is a probability-$q$ zero-chain.
Moreover, the definition of $\rho$ implies $\Norm{\rho_R(\vx/b)}\le R$ for every $\vx\in\BR^d$.
Note that the proof of Lemma 5 of \citet{arjevani2023lower} only requires the oracle is a probability zero-chain, so that we can use this lemma with $K=1$ to achieve 
\[
    \operatorname{prog}_{{1}/{4}}\big(\mU^{\top}\rho_{R}(\vx_{\mathsf{A}}^{(t)}/b) \big)< T
\]
for all $t\le(T-\log(2/\delta))/(2q)$.
Hence, we have
\[
    \left|
    [\mU^\top\rho_R(\vx^{(t)}/b)]_T
    \right|
    \le\frac14.
\]

According to Lemma~\ref{lem:global-hard-properties},
we obtain
\[
    \Norm{\nabla H_{\mU,\mu}(\vx^{(t)}/b)}
    \ge\frac14.
\]
Based on the definition  
\[
    \nabla f_{\mU}(\vx)
    =
    \frac ab\nabla H_{\mU,\mu}(\vx/b)
\]
with $a/b=16\epsilon$, it follows that
\[
    \Norm{\nabla f_{\mU}(\vx^{(t)})}
    \ge4\epsilon.
\]
This completes the proof.
\end{proof}

\subsection{The Proof of Theorem~\ref{thm:L1-lower-bound}}
\begin{proof}

We follow the settings of $a$, $b$, $\mu$, and $R$ by equations (\ref{eq:hard-scaling-parameters}) and (\ref{eq:global-hard-function}).
For $\epsilon\leq\hat c(\sqrt{L_0\Delta}+L_1\Delta)$ with $\hat c=1/483329$, we let
\begin{equation}
\begin{split}
    T:=\left\lfloor\frac{\Delta}{2a}\right\rfloor=\left\lfloor\frac{(L_0+L_1\epsilon)\Delta}{60416\epsilon^2} \right\rfloor\geq8.
\end{split}
\label{eq:calibrated-chain-length}
\end{equation}
Indeed, the condition on $\epsilon$ implies
\begin{align}\label{eq:L0-L1-epsilon}
    \frac{(L_0+L_1\epsilon)\Delta}{\epsilon^2}=\frac{L_0\Delta}{\epsilon^2}+\frac{L_1\Delta}{\epsilon}\geq\frac{\sqrt{L_0\Delta}}{\epsilon}+\frac{L_1\Delta}{\epsilon}-\frac14\geq483328,
\end{align}
where the first inequality follows from the fact $(z-1/2)^2\geq0$ for all $z\geq0$.

According to Lemma~\ref{lem:hard-scaling}(a) and (b), the function $f_{\mU}$ is $(L_0,L_1)$-smooth and satisfies
\begin{equation*}
    f_{\mU}(\vzero)
    -\inf_{\vx\in\BR^d}f_{\mU}(\vx)
    \le2aT\le\Delta,
    %\label{eq:calibrated-gap-bound}
\end{equation*}
where the last inequality follows from \eqref{eq:calibrated-chain-length}. Thus, we conclude $f_{\mU}\in\mathcal F(L_0,L_1,\Delta)$.

Following the setting of $q$ in \eqref{eq:revelation-probability},
Lemma~\ref{lem:oracle-validity} implies the oracle $\mathsf O_{f_{\mU}}(\vx;Z)$  is unbiased and has a $p$th bounded central moment.

We let $\mU\sim\operatorname{Unif}(\operatorname{St}(d,T))$ and use $P_f$ to denote the induced distribution of $f_{\mU}$. 
We take sufficiently large $d$ such that
\begin{equation*}
    d\ge
    \left\lceil
    T+32R^2\log\left(\frac{4T^2}{q}\right)
    \right\rceil,
    %\label{eq:calibrated-rotation-dimension}
\end{equation*}
which satisfies \eqref{eq:rotation-dimension} with $\delta=1/2$.
According to Lemma~\ref{lem:randomized-hardness}, we know that for all randomized algorithms $\mathsf A\in\mathcal A_{\mathrm{rand}}$ and $t\le N$  with
\begin{equation}
    N:=
    \left\lfloor
    \frac{T-\log4}{2q}
    \right\rfloor,
    \label{eq:calibrated-query-limit}
\end{equation}
we have
\begin{equation*}
    \mathbb P\left(
    \Norm{\nabla f_{\mU}(\vx^{(t)})}
    \ge4\epsilon
    \right)
    \ge\frac12.
    %\label{eq:calibrated-gradient-probability}
\end{equation*}
This implies
\begin{align*}
    \BE\left[\Norm{\nabla f_{\mU}(\vx^{(t)})}\right]
    \ge
    4\epsilon\,
    \mathbb P\left(
    \Norm{\nabla f_{\mU}(\vx^{(t)})}
    \ge4\epsilon
    \right)\ge2\epsilon>\epsilon,
    %\label{eq:calibrated-expected-gradient}
\end{align*}
where the expectation and probability are taken over $\mU$, the oracle randomness, and the internal randomness of $\mathsf A$.
Since the distribution $P_f$ is supported on $\mathcal F(L_0,L_1,\Delta)$ and the oracle $\mathsf O_{f_{\mU}}$ satisfies condition (\ref{eq:oracle-p-moment}), the distributional complexity defined in \eqref{eq:distributional-complexity} implies
\begin{equation}
    m_{\epsilon}^{\mathrm{rand}}
    (L_0,L_1,\Delta,\sigma,p)>N.
    \label{eq:lower-via-N}
\end{equation}

We now provide the lower bound for $N$. 
According to \eqref{eq:calibrated-query-limit}, we have
\begin{align}
    N\ge\frac{T-\log4}{2q}-1=\frac{T-\log4-2q}{2q}\ge\frac{T-4}{2q}\ge\frac{T}{4q},
    \label{eq:calibrated-N-lower}
\end{align}
where the first inequality uses $\lfloor x\rfloor\ge x-1$, and the last two inequalities follow from $q\le1$, $\log4<2$, and $T\ge8$.
Moreover, by equations (\ref{eq:calibrated-chain-length})
and (\ref{eq:L0-L1-epsilon}),
we obtain
\begin{equation}
    T\geq\frac{(L_0+L_1\epsilon)\Delta}{60416\epsilon^2}-1\geq\frac{(L_0+L_1\epsilon)\Delta}{120832\epsilon^2}.
\label{eq:calibrated-T-lower}
\end{equation}
Additionally, the definition of $q$ implies
\begin{align}
\begin{split}    
    \frac1q
    =\max\left\{
    1,\,  2^{-\frac1{p-1}}
    (3984)^{-\frac p{p-1}}\cdot
    \frac{\sigma^{\frac p{p-1}}}{\epsilon^{\frac p{p-1}}}
    \right\}
    \ge \frac{
    2^{-\frac1{p-1}}
    (3984)^{-\frac p{p-1}}
    }{2}
    \left(\frac{\sigma^{\frac p{p-1}}}{\epsilon^{\frac p{p-1}}}+1
    \right),
\end{split}\label{eq:calibrated-q-lower}
\end{align}
where the inequality follows from $\max\{u,v\}\ge(u+v)/2$ for $u,v\ge0$ and
$2^{-\frac1{p-1}}(3984)^{-\frac p{p-1}}\le1$. For any given~$p\in(1,2]$, combining equations (\ref{eq:lower-via-N}), (\ref{eq:calibrated-N-lower}), (\ref{eq:calibrated-T-lower}),  and (\ref{eq:calibrated-q-lower}) yields
\begin{align*}
    &m_{\epsilon}^{\mathrm{rand}}
    (L_0,L_1,\Delta,\sigma,p)\geq\Omega\left(
    \left(\frac{\sigma^{\frac p{p-1}}}{\epsilon^{\frac p{p-1}}}+1\right)
    \left(\frac{L_0\Delta}{\epsilon^2}+\frac{L_1\Delta}{\epsilon}\right)
    \right).
\end{align*}
This finishes the proof.
\end{proof}

\section{The Proof of the Result in Section \ref{sec:upper-bound}}

This section provides the proof for the upper bound of Algorithm \ref{alg:sn-mb-sgd}.
For a given $\gamma\geq 0$, we define
\[
\varphi_\gamma(\vv)= \begin{cases}
\dfrac{\vv}{\|\vv\|+\gamma},& \text{if}~\|\vv\|+\gamma>0,\\[4mm]
\vzero,& \text{if}~\vv=\vzero~~\text{and}~~\gamma=0.
\end{cases}
\]
We let $\mathcal F_0:=\sigma(x_0)$ and $\mathcal F_t:=\sigma(x_0, \{Z_{0,i}\}_{i=1}^{b}, \dots, \{Z_{t-1,i}\}_{i=1}^{b})$ for $t\ge1$
where $\sigma(\cdot)$ denotes the~$\sigma$-algebra generated by its arguments. Then $\{\mathcal F_t\}_{t\ge0}$ is a filtration representing the information available immediately before the $t$th iteration of Algorithm \ref{alg:sn-mb-sgd}.

\subsection{Some Technical Lemmas}

We provide some technical lemmas for the analysis of normalized SGD.

\begin{lemma}[{\citet[Appendix A.1]{zhang2020improved}}]\label{lem:sym-smooth}
If $f$ is $(L_0,L_1)$-smooth, then it holds
\begin{align}
    f(\vy) \leq f(\vx) + \inner{\nabla f(\vx)}{\vy-\vx} + \frac{L_0+L_1\Norm{\nabla f(\vx)}}{2}\Norm{\vy-\vx}^2
\end{align}
for all $\vx,\vy\in\BR^d$ with $\Norm{\vy-\vx}\leq 1/L_1$.
\end{lemma}

\begin{lemma}\label{lem:normalize}
For all $\gamma \geq 0$, $\vg,\vh\in\BR^d$, it holds
\begin{align*}
\langle \vh,\varphi_\gamma(\vg)\rangle
\ge \|\vh\|-2\|\vg-\vh\|-\gamma.        
\end{align*}
\end{lemma}
\begin{proof}
If $\vg=\vzero$ and $\gamma=0$, the result is trivial. Hence, we focus on the case of $\|\vg\|+\gamma>0$.
According to the Cauchy–Schwarz inequality and the triangle inequality, we achieve 
\[
\begin{aligned}
\langle \vh,\varphi_\gamma(\vg)\rangle
&=\frac{\|\vg\|^2+\langle \vh-\vg,\vg\rangle}{\|\vg\|+\gamma}\\
&\ge \frac{\|\vg\|^2}{\|\vg\|+\gamma}
      -\frac{\|\vh-\vg\|\cdot\|\vg\|}{\|\vg\|+\gamma}\\
&\ge \|\vg\|-\gamma-\|\vh-\vg\|\\
&\ge \|\vh\|-\gamma-2\|\vh-\vg\|.
\end{aligned}
\]
This finishes the proof.
\end{proof}

\begin{lemma}
For Algorithm \ref{alg:sn-mb-sgd}, it holds that
\begin{align*}
    \BE[\Norm{\vg_t-\nabla f(\vx_t)}] \leq   \frac{2 \sqrt{2}\sigma}{b^{1-1 / p}}.
\end{align*}
\end{lemma}

\begin{proof}
Recall that the stochastic gradient evaluation in the algorithm is 
\begin{align*}
    \vg_t = \frac{1}{b}\sum_{i=1}^b  \vg(\vx_t;Z_{t,i}),
    \qquad\text{where}~~ Z_{t,1},\dots,Z_{t,b}\overset{\mathrm{i.i.d}}\sim \fZ.
\end{align*}
Hence, it holds
\begin{align}
\begin{split}    
     &   \mathbb{E}\left[\Norm{\vg_t-\nabla f(\vx_t)}\right] \\
     = & \mathbb{E}\left[\Norm{\sum_{i=1}^{b} \frac{1}{b}\left(\vg(\vx_t;Z_{t,i})-\nabla f(\vx_t)\right)}\right]  \\
    \leq & 2 \sqrt{2} \mathbb{E}\left [\left(\frac{1}{b^{p}} \sum_{i=1}^{b}\Norm{\vg(\vx_t;Z_{t,i})-\nabla f(\vx_t)}^{p}\right)^{1 / p}\right] \\
    = &\frac{2 \sqrt{2}}{b} \mathbb{E}\left[\left(\sum_{i=1}^{b}\Norm{\vg(\vx_t;Z_{t,i})-\nabla f(\vx_t)}^{p}\right)^{1 / p}\right]\\
\leq & \frac{2 \sqrt{2}}{b} \left(\mathbb{E}\left[\sum_{i=1}^{b}\Norm{\vg(\vx_t;Z_{t,i})-\nabla f(\vx_t)}^{p}\right]\right)^{1 / p}\\
   \leq  & \frac{2 \sqrt{2}}{b} \cdot b^{1 / p} \sigma= \frac{2 \sqrt{2}\sigma}{b^{1-1 / p}},
\end{split}
\end{align}
where the first inequality is based on Lemma 4.3 of  \citet{liu2025nonconvex},
second inequality is based on Jensen's inequality, and the last inequality is based on $p$-BCM assumption.
\end{proof}

The update of the algorithm implies
\begin{align}\label{eq:descent-f-x}
\begin{split}    
f(\vx_{t+1})
&\le f(\vx_t)-\eta\langle \nabla f(\vx_t),\varphi_\gamma(\vg_t)\rangle
+\frac{L_0+L_1\|\nabla f(\vx_t)\|}{2} \cdot \|\vx_{t+1}-\vx_t\|^2\\
&\le f(\vx_t)-\eta(\|\nabla f(\vx_t)\| - \gamma - 2\|\vg_t-\nabla f(\vx_t)\|) +\frac{\eta^2(L_0+L_1\|\nabla f(\vx_t)\|)}{2} \\ 
&= f(\vx_t)
-\eta\left(1-\frac{\eta L_1}{2}\right)\|\nabla f(\vx_t)\|
+2\eta\|\vg_t-\nabla f(\vx_t)\|
+\eta\gamma+\frac{\eta^2L_0}{2},
\end{split}
\end{align}
where the first step is based on Lemma \ref{lem:sym-smooth}, the second step is based on Lemma \ref{lem:normalize} with $\vg=\vg_t$ and~$\vh=\nabla f(\vx_t)$ and the fact $\eta L_1 \leq 1$ from the setting of $\eta=\epsilon/(4(L_0+L_1\epsilon))$. 

Taking the expectation on \eqref{eq:descent-f-x}, we have
\begin{align*}
\mathbb E[f(\vx_{t+1})\mid\mathcal F_t]
\le f(\vx_t)
-\eta\left(1-\frac{\eta L_1}{2}\right)\|\nabla f(\vx_t)\|
+\eta\left(\frac{4\sqrt{2}\sigma}{b^{1-1 / p}}+\gamma+\frac{\eta L_0}{2}\right).
\end{align*}
Averaging on the above result over $t=0,\dots,T-1$ and taking the total expectation, we achieve
\begin{align}\label{eq:exp-grad}    
\mathbb E\left[\|\nabla f(\widehat \vx)\|\right]
\le
\frac{
\dfrac{\Delta}{\eta T}
+\gamma
+\dfrac{4\sqrt2\,\sigma}{b^{(p-1)/p}}
+\dfrac{\eta L_0}{2}
}{
1-\dfrac{\eta L_1}{2}
}.
\end{align}
Note that the parameter settings in (\ref{eq:nsgd-parameter}) indicates
\begin{align*}
\frac{\Delta}{\eta T}\le\frac{\epsilon}{4},
\qquad
\gamma \leq \frac{\epsilon}{4},
\qquad
\dfrac{4\sqrt2\,\sigma}{b^{(p-1)/p}}\le\frac{\epsilon}{8},
\qquad
\frac{\eta L_0}{2}\le\frac{\epsilon}{8},
\qquad
\text{and}
\qquad
\eta L_1\le\frac14.
\end{align*}
Combining the above inequalities with \eqref{eq:exp-grad}, we obtain $\mathbb E\left[\|\nabla f(\widehat \vx)\|\right] \leq \epsilon$.

\section{Discussion on the Target Accuracy}\label{appendix:accuracy}

In the case of $\epsilon\geq\Omega(\sqrt{L_0\Delta}+ L_1\Delta)$, we can show that the initial point $\vx_0$ is already an~$\fO(\epsilon)$-stationary point of the~$(L_0, L_1)$-smooth function $f(\cdot)$. 
Note that Lemma~\ref{lem:sym-smooth} with
$\vx=\vx_0$,
 $\vy=\vx_0-\eta\nabla f(\vx_0)$, 
 and
 $\eta=1/(L_0+L_1\Norm{\nabla f(\vx_0)})$ implies
\begin{align*}
  \Delta & \ge f(\vx_0)-\inf_{\vx\in\BR^d}f(\vx) \ge f(\vx_0)-f(\vy)\\
 &\ge -\inner{\nabla f(\vx_0)}{\vy-\vx_0}
     -\frac{L_0+L_1\Norm{\nabla f(\vx_0)}}{2}
     \Norm{\vy-\vx_0}^2 \\
 &\ge \eta \Norm{\nabla f(\vx_0)}^2-\frac{L_0+L_1\Norm{\nabla f(\vx_0)}}{2}\cdot\eta^2\Norm{\nabla f(\vx_0)}^2\\
 &=\frac{\Norm{\nabla f(\vx_0)}^2}{2(L_0+L_1\Norm{\nabla f(\vx_0)})}.
 \end{align*}
 Rearranging the above inequality yields
 \begin{equation*}
     \Norm{\nabla f(\vx_0)}^2-2L_1\Delta \Norm{\nabla f(\vx_0)}-2L_0\Delta\le0,
 \end{equation*}
 which implies
 $\Norm{\nabla f(\vx_0)}
  \le  L_1\Delta+\sqrt{L_1^2\Delta^2+2L_0\Delta}$. 
 %= \fO\left(\sqrt{L_0\Delta}+L_1\Delta\right).
 Hence, we directly obtain $\Norm{\nabla f(\vx_0)}\le\epsilon$ whenever
 \begin{align*}
 \epsilon\ge L_1\Delta+\sqrt{L_1^2\Delta^2+2L_0\Delta} = \Omega\big(\sqrt{L_0\Delta}+L_1\Delta\big).
 \end{align*}

\end{document}